%% file: supercontinuite20.tex
\documentclass{amsart}
\usepackage[T1]{fontenc}
\usepackage[utf8]{inputenc}
\usepackage{lmodern}
\usepackage{amsmath,amssymb,amsthm,bbm}
\usepackage[english]{babel}
\usepackage[final]{showkeys}
\usepackage{verbatim}
\usepackage{graphicx, color}
\usepackage{bm}
\DeclareMathOperator{\Animals}{\mathcal{A}nimals}
\DeclareMathOperator{\Bad}{Bad}

\newcommand{\Z}{\mathbb{Z}}

\newcommand{\Zd}{\mathbb{Z}^d}
\newcommand{\C}{\mathbb{C}}

\newcommand{\R}{\mathbb{R}}
\newcommand{\Rd}{\mathbb{R}^d}

\renewcommand{\P}{\mathbb{P}}

\newcommand{\Ed}{\mathbb{E}^d}

\newcommand{\eps}{\varepsilon}

\def\PP{\mathbb{P}}
\def\RRp{[0, +\infty)}
\def\EE{\mathbb{E}}
\def\NN{\mathbb{N}}
\def\RR{\mathbb{R}}
\def\ZZ{\mathbb{Z}}
\def\QQ{\mathbb{Q}}
\def\SS{\mathbb{S}}
\def\wT{\widetilde{T}}

\def\wx{\widetilde{x}}
\def\wkx{\widetilde{kx}}
\def\wnx{\widetilde{nx}}
\def\wy{\widetilde{y}}
\def\w0{\widetilde{0}}

\def\B{\mathcal{B}}
\def\C{\mathcal{C}}
\def\ind{{\mathbbm{1}}_}
\def\fG{\mathfrak{G}}
\def\fH{\mathfrak{H}}

\renewcommand{\epsilon}{\varepsilon}
\renewcommand{\phi}{\varphi}
\renewcommand{\limsup}{\overline{\lim}}
\renewcommand{\liminf}{\underline{\lim}}

\newcommand{\ie}{\emph{i.e. }}

\newcommand{\miniop}[3]{%
\renewcommand{\arraystretch}{0.6}
\begin{array}{c}
{\scriptstyle #1}\\
#2\\
{\scriptstyle #3}
\end{array}
\renewcommand{\arraystretch}{1}}

\newcommand{\Card}[1]{\vert #1 \vert}

\newcommand{\1}{1\hspace{-1.3mm}1}

\newcommand{\lestoch}{\preceq}
\newcommand{\gestoch}{\succeq}

\newtheorem{theorem}{Theorem}[section]

\newtheorem{lemme}[theorem]{Lemma}
\newtheorem{defi}[theorem]{Definition}
\newtheorem{coro}[theorem]{Corollary}
\newtheorem{prop}[theorem]{Proposition}
\newtheorem{rem}[theorem]{Remark}

\newcommand{\pr}{\mathbb{P}}

\newcommand{\bae}{\begin{equation}\begin{aligned}}
\newcommand{\eae}{\end{aligned}\end{equation}}
\newcommand{\BZ}{\mathbb{Z}}

\DeclareFontFamily{OML}{rsfs}{\skewchar\font'177}
\DeclareFontShape{OML}{rsfs}{m}{n}{ <5> <6> rsfs5 <7> <8> <9>
rsfs7 <10> <10.95> <12> <14.4> <17.28> <20.74> <24.88> rsfs10 }{}
\DeclareMathAlphabet{\mathfs}{OML}{rsfs}{m}{n}

\begin{document}

\title[Continuity of the time and isoperimetric constants]{Continuity of the time and isoperimetric constants in supercritical percolation}

{
\author[O. Garet]{Olivier Garet}
\address{Universit\'e de Lorraine, Institut \'Elie Cartan de Lorraine, UMR 7502, Vandoeuvre-l{\`e}s-Nancy, F-54506, France\\
\and \\
CNRS, Institut \'Elie Cartan de Lorraine, UMR 7502, Vandoeuvre-l{\`e}s-Nancy, F-54506, France\\
}
\email{Olivier.Garet@univ-lorraine.fr}
\author[R. Marchand]{R{\'e}gine Marchand}
\address{Universit\'e de Lorraine, Institut \'Elie Cartan de Lorraine, UMR 7502, Vandoeuvre-l{\`e}s-Nancy, F-54506, France\\
\and \\
CNRS, Institut \'Elie Cartan de Lorraine, UMR 7502, Vandoeuvre-l{\`e}s-Nancy, F-54506, France\\
}
\email{Regine.Marchand@univ-lorraine.fr}

\author[E.B. Procaccia]{Eviatar B. Procaccia}
\address{Texas A\&M University, Mailstop 3368, College Station TX 77843, USA\footnote{Research supported by NSF grant 1407558}}
\email{procaccia@math.tamu.edu}

\author[M. Th{\'e}ret]{Marie Th{\'e}ret}
\address{LPMA, Universit\'e Paris Diderot, 5 rue Thomas Mann, 75205 Paris Cedex 13, France}
\email{marie.theret@univ-paris-diderot.fr}
}

\def\motsclefs{continuity, first-passage percolation, time constant, isoperimetric constant.}


\begin{abstract}
We consider two different objects on supercritical Bernoulli percolation on the edges of $\ZZ^d$ : the time constant for i.i.d. first-passage percolation (for $d\geq 2$) and the isoperimetric constant (for $d=2$). We prove that both objects are continuous with respect to the law of the environment. More precisely we prove that the isoperimetric constant of supercritical percolation in $\BZ^2$ is continuous in the percolation parameter. As a corollary we obtain that normalized sets achieving the isoperimetric constant are continuous with respect to the Hausdroff metric. Concerning first-passage percolation, equivalently we consider the model of i.i.d. first-passage percolation on $\ZZ^d$ with possibly infinite passage times: we associate with each edge $e$ of the graph a passage time $t(e)$ taking values in $[0,+\infty]$, such that $\PP[t(e)<+\infty] >p_c(d)$. We prove the continuity of the time constant with respect to the law of the passage times. This extends the continuity property of the asymptotic shape previously proved by Cox and Kesten \cite{MR588407,MR633228,Kesten:StFlour} for first-passage percolation with finite passage times.
\end{abstract}

\maketitle


\section{Introduction}

We consider supercritical bond percolation on $\ZZ^d$, with parameter $p>p_c(d)$, the critical parameter for this percolation. Almost surely, there exists a unique infinite cluster $\C_{\infty}$ -- see for instance Grimmett's book \cite{grimmett-book}. We study the continuity properties of two distinct objects defined on this infinite cluster: the isoperimetric (or Cheeger) constant, and the asymptotic shape (or time constant) for an independent first-passage percolation. In this section, we introduce briefly the studied objects and state the corresponding results: more precise definitions will be given in the next section.


\subsection{Isoperimetric constant of the infinite cluster in dimension 2}

For a finite graph $\gimel=(V(\gimel),E(\gimel))$, the isoperimetric constant is defined as 
\[
\varphi_\gimel=\min\left\{\frac{|\partial A|}{|A|}:A\subset V(\gimel),0<|A|\le\frac{|V(\gimel)|}{2}\right\}
,\] where $\partial A$ is the edge boundary of $A$, $\partial A=\{e=(x,y)\in E(\gimel):x\in A,y\notin A, \text{ or }x\notin A,y\in A\}$, and  $|B|$ denotes the cardinal of the finite set $B$.

We consider the isoperimetric constant $\varphi_n(p)$ of $\C_{\infty} \cap [-n,n]^d$, the intersection of the infinite component of supercritical percolation of parameter $p$ with the box $[-n,n]^d$:
\[
\varphi_n(p)=\min\left\{\frac{|\partial A|}{|A|}:A\subset \C_{\infty} \cap [-n,n]^d,0<|A|\le\frac{|\C_{\infty} \cap [-n,n]^d|}{2}\right\}
,\]
 In several papers (e.g. \cite{benjamini2003mixing}, \cite{mathieu2004isoperimetry}, \cite{ECP1390},   \cite{berger2008anomalous}), it was shown that there exist constants $c,C>0$ such that $c< n\varphi_n(p)<C$, with probability tending rapidly to $1$. This led Benjamini to conjecture the existence of $ \lim_{n\rightarrow+\infty}n\varphi_n(p)$. In \cite{procaccia2011concentration},  Rosenthal and Procaccia proved that the variance of $n\varphi_n (p)$  is smaller than $Cn^{2-d}$, which implies $n  \varphi_n(p)$ is concentrated around its mean for $d\ge 3$. In \cite{biskup2012isoperimetry},  Biskup, Louidor, Procaccia and Rosenthal proved the existence of $ \lim_{n\rightarrow+\infty}n\varphi_n(p)$ for $d=2$. This constant is called the Cheeger constant. In addition, a  shape theorem was obtained: any set yielding the isoperimetric constant converges in the Hausdorff metric to the normalized Wulff shape $\widehat{W}_p$, with respect to a specific norm given in an implicit form, see Proposition \ref{theo:minimizer} below. For additional background and a wider introduction on Wulff construction in this context, the reader is referred to~\cite{biskup2012isoperimetry}. 
Our first result is the continuity of the Cheeger constant and of the Wulff shape in dimension $d=2$:
\begin{theorem}
\label{thm:isocnt}
For $d=2$, the applications
$$p\in(p_c(2),1] \mapsto \lim_{n\rightarrow+\infty}n\varphi_n (p) \quad  \quad \text{ and } \quad \quad p\in(p_c(2),1] \mapsto \widehat{W}_p$$
are continuous, the last one for the Hausdorff distance between non-empty compact sets of $\R^2$.
\end{theorem}


\subsection{First-passage percolation on the infinite cluster in dimension $d\ge 2$}

Consider a fixed dimension $d\ge 2$. First-passage percolation on $\ZZ^d$ was introduced by Hammersley and Welsh~\cite{HW} as a model for the spread of a fluid in a porous medium. To each edge of the $\Zd$ lattice is attached a nonnegative random variable $t(e)$ which corresponds to the travel time needed by the fluid to cross the edge. When the passage times are independent identically distributed variables with common distribution $G$, with suitable moment conditions, 
the time needed to travel from $0$ to $nx$ is equivalent to $n\mu_G(x)$, 
 where  $\mu_G$ is a semi-norm associated to $G$ called the time constant;  Cox and Durrett~\cite{MR624685} proved this result under necessary and sufficient integrability conditions on the distribution $G$ of the passage times. Kesten in \cite{kesten1986aspects} proved that the semi-norm $\mu_G$ is a norm if and only if $G(\{0\} )<p_c(d)$. In casual terms, the asymptotic shape theorem (in its geometric form) says that in this case, the random ball of radius $n$, \ie the set of points that can be reached within time $n$ from the origin, asymptotically looks like $n \B_{\mu_G}$, where $\B_{\mu_G}$ is the unit ball for the norm  $\mu_G$. The ball $\B_{\mu_G}$ is thus called the asymptotic shape associated to $G$.

A natural extension is to replace the $\Zd$ lattice by a random environment given by the infinite cluster $\C_{\infty}$ of a supercritical Bernoulli percolation model. This is equivalent to allow $t(e)$ to be equal to $+\infty$.
The existence of a time constant in first-passage percolation in this setting was first proved by Garet and Marchand in \cite{GaretMarchand04}, in the case where $(t(e)\1_{t(e)<+\infty})$ is a stationary integrable ergodic field.
Recently, Cerf and Th\'eret~\cite{CerfTheret14} focused of the case where 
$(t(e)\1_{t(e)<+\infty})$ is an independent field, and managed to prove the existence of an appropriate time constant without any integrability assumption.
In the following, we adopt the settings of Cerf and Th\'eret: the passage times are independent random variables with common distribution $G$ taking its values in $[0,+\infty]$ such that $G([0,+\infty))>p_c(d)$, and we denote by $\mu_G$ the corresponding time constant. 

Our second result is the continuity of the time constants $\mu_G(x)$ with respect to the distribution $G$ of the passage times, uniformly in the direction.
More precisely, let $(G_n)_{n\in \NN}$ and $G$ be probability measures on $[0,+\infty]$.
We say that $G_n$ converges weakly towards $G$ when $n$ goes to infinity, and we write $G_n \overset{d}{\rightarrow} G$, if for any continuous bounded function $f: [0,+\infty] \mapsto \RRp$ we have
$$ \lim_{n \rightarrow +\infty} \int_{[0,+\infty]} f \, dG_n \,=\,  \int_{[0,+\infty]} f \, dG \,. $$
Equivalently, $G_n \overset{d}{\rightarrow} G$ if and only if $\lim_{n \rightarrow \infty} G_n([t,+\infty]) = G([t,+\infty])$ for all $t\in \RRp$ such that $t\mapsto G([t,+\infty])$ is continuous at $t$. 
Let $\SS^{d-1} = \{ x\in \RR ^d \,:\, \|x\|_2 = 1 \}$. 
\begin{theorem}
\label{thmcont}
Let $G, (G_n)_{n\in \NN}$ be probability measures on $[0,+\infty]$ such that for every $n\in \NN$, $G_n(\RRp) > p_c(d)$ and $G(\RRp) > p_c(d)$ . If $G_n \overset{d}{\rightarrow} G$, then
$$ \lim_{n\rightarrow +\infty} \sup_{x \in \SS^{d-1}} | \mu_{G_n}(x) - \mu_G(x) | \,=\, 0 \,. $$
\end{theorem}
This result extends the continuity of the time constant in classical first-passage percolation proved by Cox and Kesten \cite{MR588407,MR633228,Kesten:StFlour}  to first-passage percolation with possibly infinite passage times. As in the classical case, the semi-norm $\mu_G$ is a norm if and only if $G(\{0\})<p_c(d)$ (see proposition \ref{thmpropmu} below). In that case, we denote by $\B_{\mu_{G}}$ its unit ball and call it the asymptotic shape associated to $G$.
We can quite easily deduce from Theorem \ref{thmcont} the following continuity of the asymptotic shapes when they exist: 
\begin{coro}
Let $G, (G_n)_{n\in \NN}$ be probability measures on $[0,+\infty]$ such that for every $n\in \NN$, $G_n(\RRp) > p_c(d)$, $G(\RRp) > p_c(d)$ and $G(\{0\}) <p_c(d)$. If $G_n \overset{d}{\rightarrow} G$, then
$$ \lim_{n\rightarrow +\infty}  d_H (\B_{\mu_{G_n}}, \B_{\mu_{G}}) \,=\, 0 \,, $$
where $d_H$ is the Hausdorff distance between non-empty compact sets of $\Rd$.
\end{coro}

Particularly, when $G_p=p\delta_1+(1-p)\delta_{+\infty}$, the norm $\mu_{G_p}$ governs the asymptotic distance in the infinite cluster of a supercritical Bernoulli percolation (see \cite{GaretMarchand04,GaretMarchand07,GaretMarchand10}). We get the following corollary:

\begin{coro}
For $p>p_c(d)$, let us denote by $\mathcal{B}_p$ the unit ball for the norm that is associated to the cheminal distance in supercritical bond percolation with parameter $p$.
Then,
$$p \in (p_c(d), 1] \mapsto \mathcal{B}_p$$ 
is continuous for  the Hausdorff distance between non-empty compact sets of $\Rd$.
\end{coro}

As a key step of the proof of Theorem~\ref{thmcont}, we study the effect of truncations of the passage time on the time constant. Let $G$ be a probability measure on $[0, +\infty]$ such that $G(\RRp) >p_c(d)$.
For every $K>0$, we set
$$G^K  \,=\, \ind{[0,K)} G + G([K, +\infty]) \delta_{K} \,,$$
i.e., $G^K$ is the law of the truncated passage time $t_G^K(e) = \min (t_G(e), K)$. We have the following control on the effect of these truncations on the time constants:

\begin{theorem}
\label{propetape2}
Let $G$ be a probability measure on $[0, +\infty]$ such that $G(\RRp) >p_c(d)$. Then
$$ \forall x\in \ZZ^d \quad \lim_{K\rightarrow \infty}  \mu_{G^K}(x) \,=\,  \mu_G (x)\,.$$
\end{theorem}

As a consequence of these results, we can approximate the time constants for the chemical distance in supercritical percolation on $\ZZ^d$ by the time constants for some finite passage times:
\begin{coro}
Let $p>p_c(d)$, and consider $G = p \delta_1 + (1-p) \delta_{+\infty}$.  Then $G^K = p \delta_1 + (1-p) \delta_{K}$ for all $K\geq 1$ and
$$  \forall x\in \ZZ^d\quad  \lim_{K\rightarrow \infty} \mu_{G^K} (x) \,=\, \mu_G (x) \,. $$
\end{coro}


\subsection{Idea of the proofs}

Obviously, the two main theorems of the paper, Theorems \ref{thm:isocnt} and \ref{thmcont}, state results of the same nature. Beyond this similarity, their proofs share a common structure and a common renormalisation step. The idea of the delicate part of both proofs is inspired by Cox and Kesten's method in \cite{MR633228}. Consider that some edges of $\ZZ^d$ are "good" (i.e. open, or of passage time smaller than some constant), and the others are "bad", for a given law of the environment (a parameter $p$ for the percolation, or a given law $G$ of passage times), and look at a path of good edges in this setting. Then change a little bit your environment : decrease $p$ to $p-\epsilon$, or increase the passage times of the edges. Some edges of the chosen path become bad. To recover a path of good edges, you have to bypass these edges. The most intuitive idea is to consider the cluster of bad edges around each one of them, and to bypass the edge by a short path along the boundary of this cluster. This idea works successfully in Cox and Kesten's paper. Unfortunately in our setting the control we have on these boundaries, or on the number of new bad edges we create, is not good enough. This is the reason why we cannot perform our construction of a modified good path at the scale of the edges. Thus we need to use a coarse graining argument to construct the bypasses at the scale of good blocks.

In section \ref{secdef}, we give more precise definitions of the studied objects and state some preliminary results. In Section \ref{secrenorm}, we present the renormalization process and the construction of modified paths that will be useful to study both the time constant and the isoperimetric constant. Sections \ref{secetape2} and \ref{seccontfpp} are devoted to the study of first-passage percolation. In Section \ref{secetape2}, we use the renormalization argument to study the effect of truncating the passage times on the time constant. We then use it in Section \ref{seccontfpp} to prove the continuity of the time constant. Finally Section \ref{seccheeger} is devoted to the proof of the continuity of the isoperimetric constant, using again the renormalization argument.


\section{Definitions and preliminary results}
\label{secdef}

In this section we give a formal definition of the objects we briefly presented in the introduction. We also present the coupling that will be useful in the rest of the paper, and prove the monotonicity of the time constant.


\subsection{Lattice and passage times}

Let $d\geq 2$. We consider the graph whose vertices are the points of $\ZZ^d$, and we put an edge between two vertices $x$ and $y$ if and only if the Euclidean distance between $x$ and $y$ is equal to $1$. We denote this set of edges by $\EE^d$. We denote by $0$ the origin of the graph. For $x=(x_1,\dots , x_d)\in \R^d$, we define $\|x\|_1 = \sum_{i=1}^d |x_i|$, $\|x\|_2 = \sqrt{\sum_{i=1}^d x_i^2}$ and $\|x\|_{\infty} = \max \{ |x_i| \,:\, i \in \{1,\dots , d\} \}$.

Let $(t(e),e\in \EE^d)$ be a family of i.i.d. random variables taking values in $[0, +\infty]$ with common distribution $G$. We emphasize that $+\infty$ is a possible value for the passage times, on the contrary to what is assumed in classical first-passage percolation. The random variable $t(e)$ is called the {\em passage time} of $e$, i.e., it is the time needed to cross the edge $e$. If $x,y$ are vertices in $\ZZ^d$, a path from $x$ to $y$ is a sequence $r=(v_0,e_1, \dots , e_n, v_n)$ of vertices $(v_i)_{i=0, \dots , n}$ and edges $(e_i)_{i=1, \dots ,n}$ for some $n\in \NN$ such that $v_0 = x$, $v_n=y$ and for all $i\in \{1,\dots ,n\}$, $e_i$ is the edge of endpoints $v_{i-1}$ and $v_i$. We define the length $|r|$ of a path $r$ as its number of edges and we define the passage time of $r$ by $T(r) = \sum_{e \in r} t(e)$. We obtain a random pseudo-metric $T$ on $\ZZ^d$ in the following way (the only possibly missing property is the separation of distinct points):
$$ \forall x,y \in \ZZ^d \,, \quad T(x,y) \,=\, \inf \{ T(r) \,:\, r \textrm{ is a path from } x \textrm{ to } y \} \in [0,+\infty]\,.$$
Since different laws appear in this article, we put a subscript $G$ on our notations to emphasize the dependance with respect to  the probability measure $G$ : $t_G(e)$, $T_G(r)$ and $T_G (x,y)$.

As we are interested in the asymptotic behavior of the pseudo-metric $T_G$,  we will only consider laws $G$ on $[0,+\infty]$ such that $G(\RRp)> p_c(d)$. Here and in the following, $p_c(d)$ denotes the critical parameter for bond Bernoulli percolation on $(\ZZ^d, \EE^d)$. Thus there a.s. exists a unique infinite cluster $\C_{G,\infty}$ in the super-critical percolation $(\ind{\{ t_G(e) <\infty \}}, e\in \EE^d)$ that only keeps edges with finite passage times. Our generalized first-passage percolation model with time distribution $G$ is then equivalent to  standard i.i.d. first-passage percolation (where the passage time of an edge $e$ is the law of $t_G(e)$ conditioned to be finite) on a super-critical Bernoulli percolation performed independently (where  the parameter for an edge to be closed is $G(\{+\infty\})$). For instance, if we take $G_p = p \delta_1 + (1-p) \delta_{+\infty}$ with $p>p_c(d)$, the pseudo-distance $T_{G_p}$
is the chemical distance in supercritical bond percolation with parameter $p$.

To get round the fact that the times $T_G$ can take infinite values,  we introduce some regularized times $\wT_G^{\C}$, for well chosen sets $\C$. These regularized passage times have better integrability properties.
Let $\C$ be a subgraph of $(\ZZ^d, \EE^d)$. Typically, $\C$ will be the infinite cluster of an embedded supercritical Bernoulli bond percolation. For every $x\in \ZZ^d$, we define the random vertex $\wx^{\C}$ as the vertex of $\C$ which minimizes $\|x-\wx^{\C} \|_1$, with a deterministic rule to break ties. We then define the regularized passages times $\wT_G^{\C}$ by
$$ \forall x,y \in \ZZ^d \,,\quad \wT_G^{\C} (x,y) \,=\, T_G (\wx^{\C}, \wy^{\C})\,. $$


\subsection{Definition of the Cheeger constant in supercritical percolation on $\Z^2$}
\label{secdefcheeger}

We collect in this subsection the definitions and properties of the Cheeger constant obtained in \cite{biskup2012isoperimetry}. The Cheeger constant can be represented as the solution of a continuous isoperimetric problem with respect to some norm. To define this norm, we first require some definitions. We fix $p>p_c(2)$, we denote by $\C_p $ the $\P_p$-a.s. unique infinite cluster $\C_{G_p, \infty}$ and we set $\theta_p=\P_p(0 \in \C_p)$. 

For a path $r=(v_0,e_1,\ldots,e_n,v_n)$, and $i\in\{2,\ldots,n-1\}$, an edge $e=(x_i,z)$ is said to be a right-boundary edge if $z$ is a neighbor of $x_i$ between $x_{i+1}$ and $x_{i-1}$ in the clockwise direction. The right boundary $\partial^+r$ of $r$ is the set of right-boundary edges. A path is called right-most if it uses every edge at most once in every orientation and it doesn't contain right-boundary edges. See Figure \ref{fig:rmp}; the solid lines represent the path, dashed lines represent the right-boundary edges, and the curly line is a path in the medial graph which shows the orientation (see  \cite{biskup2012isoperimetry} for a thorough discussion). For $x,y\in\mathbb{Z}^2$, let $\mathcal{R}(x,y)$ be the set of right-most paths from $x$ to $y$. For a path $r\in\mathcal{R}(x,y)$, define ${\bf b}(r)=|\{e\in\partial^+r:e\text{ is open}\}|$. For $x,y\in\mathcal{C}_p$ we define the right boundary distance, $b(x,y)=\inf\{{\bf b}(r):r\in\mathcal{R}(x,y), \, r\text{ is open}\}$. The next result yields uniform convergence of the right boundary distance to a norm on~$\mathbb{R}^2$.
\begin{figure}
\includegraphics[width=3in]{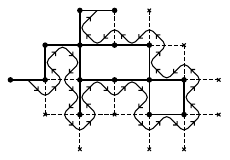}
\caption{A right most path}
\label{fig:rmp}
\end{figure}

\begin{prop}[Definition of the norm, Theorem 2.1 in \cite{biskup2012isoperimetry} ]
For any $p>p_c(2)$, there exists a norm $\beta_p$ on~$\mathbb{R}^2$ such that for any $x \in \mathbb{R}^2$, 
\begin{equation}\nonumber
        \beta_p(x)
        := \lim_{n \rightarrow \infty} \frac{b(\widetilde{0}^{\C_p}, \widetilde{nx}^{\C_p})}{n} \quad \quad \P_p-a.s. \text{ and in } L^1(\P_p).
\end{equation}
Moreover, the convergence is uniform on $\mathbb{S}^1 = \{x\in\mathbb{R}^2:\: \|x\|_2 = 1\}$. 
\end{prop}

We will require the following control on the length of right-most paths.
\begin{lemme}[Proposition 2.9 in \cite{biskup2012isoperimetry}] \label{lem:rightint}
There exist $C,C',\alpha>0$ (depending on $p$) such that for all $n$,
\[
\pr \left[\exists \gamma\in \bigcup_{x\in\mathbb{Z}^2} \mathcal{R}(0,x)\,:\,|\gamma|>n \,,\,\, \mathbf{b} (\gamma) \leq \alpha n \right] \le  Ce^{-C'n}
.\]
\end{lemme}

 The connection between the Cheeger constant and the norm $\beta_p$ goes through a continuous isoperimetric problem. For a continuous curve $\lambda:[0,1]\rightarrow\mathbb{R}^2$, and a norm $\rho$, let the $\rho$-length of $\lambda$ be 
 \[
 \text{len}_\rho(\lambda)=\sup_{N\ge 1}\sup_{0\le t_0<\ldots<t_N\le1}\sum_{i=1}^{N}\rho(\lambda(t_i)-\lambda(t_{i-1}))
 .\]
 A curve $\lambda$ is said to be rectifiable if $\text{len}_\rho(\lambda)<\infty$ for any norm $\rho$. A curve $\lambda$ is called a Jordan curve if $\lambda$ is rectifiable, $\lambda(0)=\lambda(1)$ and $\lambda$ is injective on $[0,1)$. For any Jordan curve $\lambda$, we can define its interior $\text{int}(\lambda)$ as the unique finite component of $\mathbb{R}^2\setminus\lambda([0,1])$. Denote by Leb the Lebesgue measure on $\mathbb{R}^2$. The Cheeger constant can be represented as the solution of the following continuous isoperimetric problem:
 \begin{prop}[Theorem 1.6 in \cite{biskup2012isoperimetry}]\label{thm:cheegervariational}
 \label{propbiskup}
 For every $p>p_c(2)$, 
 \[
 \lim_{n\rightarrow+\infty}n\varphi_n(p)=(\sqrt{2} \, \theta_p)^{-1}\inf\{\text{\rm len}_{\beta_p}(\lambda):\lambda\textrm{ is a Jordan curve}, \mathrm{Leb}(\mathrm{int}(\lambda))=1\}
 .\]
 \end{prop}
Moreover one obtains a limiting shape for the sets that achieve the minimum in the definition of $\varphi_n(p)$. This limiting shape is given by the Wulff construction \cite{wulff1901xxv}. Denote by 
\begin{equation}\label{eq:wulff}
 W_p=\bigcap_{\hat{n}:\|\hat{n}\|_2=1}\{x\in\mathbb{R}^2:\hat{n}\cdot x\le\beta_p(\hat{n})\} \text{ and } \widehat{W}_p=\frac{W_p}{\sqrt{\text{Leb}(W_p)}},
\end{equation}
where $\cdot$ denotes the Euclidean inner product. The set $\widehat{W}_p$ is a minimizer for the isoperimetric problem associated with the norm $\beta_p$, and it gives the asymptotic shape of the minimizer sets in the definition of $\varphi_n(p)$.
 Denote by ${\mathcal{U}}_n(p)$ be the set of minimizers of $\varphi_n(p)$; then  
 \begin{prop}[Shape theorem for the minimizers, \cite{biskup2012isoperimetry} Theorem 1.8]
\label{theo:minimizer}
For every $p > p_c(2)$, $\P_p$ almost surely, 
\begin{align*}
& \max_{U \in {\mathcal{U}}_n(p)}
        \inf_{\substack{\xi \in \mathbb{R}^2 
                        }}
            d_{\rm H} \left( \frac{U}n \,, \,\,  \xi + {\sqrt{2}\, }{\widehat{W}_p} \right) \,\,\underset{n\to+\infty}\longrightarrow\,\, 0. 
\end{align*}
\end{prop}
By Proposition \ref{thm:cheegervariational} and Definition \eqref{eq:wulff},  Theorem \ref{thm:isocnt} will follow from the continuity of $p\mapsto \beta_p$.


\subsection{Definition and properties of the time constant}

As announced in the introduction, we follow the approach  by Cerf and Th\'eret in \cite{CerfTheret14}, which requires no integrability condition on the restriction of $G$ to $\RRp$.  We collect in this subsection the definition and properties of the time constants obtained in their paper.

Let $G$ be a probability measure on $[0,+\infty]$ such that $G([0,+\infty))>p_c(d)$, and let $M>0$ be such that $G([0,M])> p_c(d)$. We denote by $\C_{G,M}$ the a.s. unique infinite cluster of the percolation $(\ind{\{ t_G(e) \leq M \}}, e\in \EE^d)$, \ie the percolation obtained by keeping only edges with passage times less than $M$. For any $x,y \in \Zd$, the (level $M$) regularized passage time $\wT_G^{\C_{G,M}}(x,y)$ is then
$$\wT_G^{\C_{G,M}}(x,y) \,=\, T_G (\wx^{\C_{G,M}}, \wy^{\C_{G,M}})\,. $$
The parameter $M$ only plays a role in the choice of $\wx^{\C_{G,M}}$ and $\wy^{\C_{G,M}}$. Once these points are chosen, the optimization in $\wT_G^{\C_{G,M}}(x,y)$ is on all paths between $\wx^{\C_{G,M}}$ and $\wy^{\C_{G,M}}$, paths using edges with passage time larger than $M$ included. But as $\wx^{\C_{G,M}} \in \C_{G,M}$ and $\wy^{\C_{G,M}}\in \C_{G,M}$, we know that exists a path using only edges with passage time less than $M$ between these two points. To be more precise, we denote by $D^{\C} (x,y)$ the chemical distance (or graph distance) between two vertices $x$ and $y$ on $\C$:
$$  \forall x,y\in \ZZ^d \,,\quad D^{\C} (x,y) \,=\, \inf \{ |r| \,:\, r \textrm{ is a path from }x \textrm{ to }y \,,\, r\subset \C \} \,,$$
where $\inf \varnothing = +\infty$. The event that the vertices $x$ and $y$ are connected in $\C$ is denoted by \smash{$\{ x \overset{\C}{\longleftrightarrow} y \}$}. Then, for any $x,y \in \Zd$,
$$\wT_G^{\C_{G,M}}(x,y) \, \le \, M  D^{\C_{G,M}} (\wx^{\C_{G,M}}, \wy^{\C_{G,M}})\,. $$
The regularized passage time $\wT_G^{\C_{G,M}}$ enjoys then  the same good integrability properties as the chemical distance on a supercritical percolation cluster (see \cite{AP}):

\begin{prop}[Moments of $\wT$, \cite{CerfTheret14}]\label{prop:integrability}
\label{propmoments}
Let $G$ be a probability measure on $[0,+\infty]$ such that $G([0,+\infty))>p_c(d)$. For every $M\in \RRp$ such that $G([0,M])> p_c(d)$, there exist positive constants $C_1, C_2$ and $C_3$ such that
$$ \forall x\in \ZZ^d \,,\,\, \forall l\geq C_3 \|x\|_1 \,,\quad \PP \left[ \wT_G^{\C_{G,M}} (0,x) > l \right] \,\leq\, C_1 e^{-C_2 l} \,.$$
\end{prop}

We denote by $\C_{G,\infty}$ the a.s. unique infinite cluster of the percolation obtained by keeping only edges with finite passage time, \ie the percolation $(\ind{\{ t_G(e) < \infty \}}, e\in \EE^d)$. Property \ref{propmoments} implies in particular that the times $\wT_G^{\C_{G,M}} (0,x)$ are integrable. A classical application of a subadditive ergodic theorem gives the existence of a time constant:

\begin{prop}[Convergence to the time constant, \cite{CerfTheret14}]
\label{thmcv}
Let $G$ be a probability measure on $[0,+\infty]$ such that $G([0,+\infty))>p_c(d)$. There exists a deterministic function $\mu_G : \Zd \rightarrow [0,+\infty)$ such that for every $M\in \RRp$ satisfying $G([0,M])> p_c(d)$, we have the following properties:
\begin{align}
& \forall x\in \ZZ^d \quad \mu_G (x) \,=\, \inf_{n\in \NN^*} \frac{\EE \left[\wT_G^{\C_{G,M}} (0,nx) \right] }{n} 
\,=\, \lim_{n\to +\infty} \frac{\EE \left[\wT_G^{\C_{G,M}} (0,nx) \right] }{n} \, , \label{thmcv-def}\\
& \forall x\in \ZZ^d \quad \lim_{n\rightarrow \infty} \frac{\wT_G^{\C_{G,M}} (0,nx)}{n} \,=\, \mu_G (x)  \quad \textrm{a.s. and in }L^1, \label{thmcv-ps}\\
& \forall x \in \ZZ^d \quad \lim_{n\rightarrow \infty} \frac{\wT_G^{\C_{G,\infty}} (0,nx)}{n} \,=\, \mu_G (x)   \quad \textrm{in probability} \,, \label{thmcv-proba}\\
& \forall x \in \ZZ^d \quad \lim_{n\rightarrow \infty} \frac{T_G (0,nx)}{n} \,=\, \theta_G^2 \delta_{\mu_G (x)} + (1-\theta_G^2) \delta_{+\infty}  \quad \textrm{in distribution}, \label{thmcv-loi}
\end{align}
where $\theta_G = \PP[0 \in \C_{G,\infty}]$.
\end{prop}
Note that even if the definition \eqref{thmcv-def}  of the time constants $\mu_G(x)$ requires to introduce a parameter $M$ in the definition of the regularized passage times $\wT_G^{\C_{G,M}}(0,nx)$, these time constants $\mu_G(x)$ do not depend on $M$. Note also that if instead of taking the $\wx^{\C_{G,M}}$  in the infinite cluster $\C_{G,M}$ of edges with passage time less than $M$, we take the $\wx^{\C_{G,\infty}}$  in the infinite cluster $\C_{G,\infty}$ of edges with finite passage time, the almost sure convergence is weakened into the convergence in probability \eqref{thmcv-proba}. Without any regularization, the convergence in \eqref{thmcv-loi} is only in law.

As in the classical first-passage percolation model, the function $\mu_G$ can be extended, by homogeneity, into a pseudo-norm on $\Rd$ (the only possibly missing property of $\mu_G$ is the strict positivity):
\begin{prop}[Positivity of $\mu_G$, \cite{CerfTheret14}]
\label{thmpropmu}
Let $G$ be a probability measure on $[0,+\infty]$ such that $G([0,+\infty))>p_c(d)$.
Then either $\mu_G$ is identically equal to $0$ or $\mu_G ( x) >0$ for all $x\neq 0$, and we know that
$$ \mu_G = 0 \, \iff \, G(\{0\}) \geq p_c(d) \,.$$
\end{prop}
Proposition \ref{propmoments} gives strong enough integrability properties of $\wT_G^{\C_{G,M}} (0,x)$ to ensure that the convergence to the time constants is uniform in the direction:

\begin{prop}[Uniform convergence, \cite{CerfTheret14}]
Let $G$ be a probability measure on $[0,+\infty]$ such that $G([0,+\infty))>p_c(d)$.
Then for every $M\in \RRp$ such that $G([0,M])> p_c(d)$, we have
$$ \lim_{n\rightarrow \infty} \,\, \sup_{x\in \ZZ^d\,,\, \|x\|_1 \geq n} \left| \frac{\wT_G^{\C_{G,M}} (0,x) - \mu_G (x)}{\|x\|_1} \right| \,=\, 0 \quad \textrm{a.s.}$$
\end{prop}

When $\mu_G>0$, this uniform convergence is equivalent to the so called {\em shape theorem}, that we briefly present now. We define $B_{G,t}$ (resp. $\widetilde{B}_{G,t}^{\C_{G,M}}$, $\widetilde{B}_{G,t}^{\C_{G,\infty}} $) as the set of all points reached from the origin within a time $t$ :
$$ B_{G,t} \,=\, \{ z\in \ZZ^d \,:\,T_G (0,z) \leq t \} \,,$$
(resp. $\wT_G^{\C_{G,M}}, \wT_G^{\C_{G,\infty}}$), and when $\mu_G$ is a norm we denote by $\B_{\mu_G}$ its closed unit ball. 
Roughly speaking, the shape theorem states that the rescaled set $B_{G,t}/t$ (respectively  $\widetilde{B}_{G,t}^{\C_{G,M}}/t$, $\widetilde{B}_{G,t}^{\C_{G,\infty}}/t $) converges  towards $\B_{\mu_G}$. The convergence holds in a sense that depends on the regularity of times considered (see \cite{CerfTheret14} for more precise results).


\subsection{Coupling}
\label{seccoupling}
To understand how $\mu_G$ depends on $G$, it is useful to consider passage times $(t_G(e))$ with common distribution $G$, that also have good coupling properties. For any probability measure $G$ on $[0,+\infty]$, we denote by $\fG$ the function 
\begin{align*}
\fG :  \RRp & \to [0,1] \\
 t & \mapsto  G([t, +\infty]),
\end{align*}
which characterizes $G$. For two probability measures $G_1, G_2$  on $[0,+\infty]$, we define the following stochastic domination relation:
$$G_1 \gestoch G_2 \quad \Leftrightarrow \quad  \forall t \in \RRp \quad  \fG_1(t)\geq \fG_2(t).$$ 
This is to have this equivalence that we choose to characterize a probability measure $G$ by $\fG$ instead of the more standard distribution function $t\mapsto G([0,t])$.

Given a probability measure $G$ on $[0,+ \infty ]$, we define the two following pseudo-inverse functions for $\fG$:
\begin{align*}
 \forall t \in [0,1]\,,\,\, & \hat \fG (t) \,=\, \sup \{ s \in \RRp \,:\, \fG(s) \geq 1-t \}   \textrm{ and } \\ & \tilde \fG (t) \,=\, \sup \{ s \in \RRp \,:\, \fG(s) > 1-t \}\,.
 \end{align*}
These pseudo-inverse functions can be used to simulate random variable with distribution $G$:
\begin{lemme}
\label{lemcoup}
Let $U$ be a random variable with uniform law on $(0,1)$. 
If $G$ is a probability measure on $[0,+ \infty ]$, then $\hat \fG(U)$ and $\tilde \fG (U)$ are random variables taking values in $[0,+\infty ]$ with distribution $G$, and $\tilde \fG(U)=\hat \fG (U)$ a.s.
\end{lemme}

\begin{proof}
The function $\hat \fG$ has the following property
\begin{equation}
\label{eqF-1}
 \forall t\in [0,1] \,,\,\, \forall s \in \RRp  \,, \quad \hat \fG(t) \geq s \iff \fG(s) \geq 1-t \,.
 \end{equation}
Then for all $s \in \RRp$, we have $\PP [\hat \fG(U) \geq s] = \PP [U \geq 1- \fG(s)] = \fG(s)$, thus $\hat \fG (U)$ has distribution $G$. The function $\tilde \fG$ does not satisfy the property (\ref{eqF-1}). However, $\hat \fG(t) \neq \tilde \fG (t)$ only for $t\in [0,1]$ such that $ \fG^{-1} (\{1-t\})$ contains an open interval, thus the set $\{ t \in [0,1] \,:\, \hat \fG(t) \neq \tilde \fG (t) \}$ is at most countable. This implies that $\hat \fG(U) = \tilde \fG (U)$ a.s., thus $\tilde \fG (U)$ has the same law as $\hat \fG(U)$.  
\end{proof}

In the following, we will always build the passage times of the edges with this lemma. Let then $(u(e), e \in \EE^d)$ be a family of i.i.d. random variables with uniform law on $(0,1)$. For any given probability measure $G$ on $[0,+\infty]$, the family of i.i.d passage times with distribution $G$ will always be 
\begin{equation}
\label{eqdeftF}
\forall e \in \EE^d\,,  \quad  t_G(e) = \hat \fG (u(e)) \,.
\end{equation}
Of course the main interest of this construction is to obtain couplings between laws: if $G_1$ and $G_2$ are probability measures on $[0,+\infty]$, 
$$G_1 \lestoch G_2 \quad \Rightarrow \quad t_{G_1} (e) \leq t_{G_2}(e) \text{ for all edges $e$}.$$
In particular in the case of Bernoulli percolation, if $p \leq q $, $G_{q} = q \delta_1 + (1-q) \delta_{\infty} \lestoch G_{p} = p \delta_1 + (1-p) \delta_{\infty}$ thus $\C_p \subset \C_q$. Moreover, we have the following pleasant property:
\begin{lemme}
\label{propCV}
Let $G, (G_n)_{n\in \NN}$ be probability measures on $[0,+\infty]$. We define the passage times $t_G(e)$ and $t_{G_n} (e)$ as in equation (\ref{eqdeftF}). If $G_n \overset{d}{\rightarrow} G$, then 
$$  \textrm{a.s.}  \,,\,\forall e\in \EE^d \,, \quad  \lim_{n\rightarrow \infty} t_{G_n} (e) \,=\, t_G (e) \,.$$
\end{lemme}

\begin{proof} $(i)$ Let us prove that if $G_n \gestoch G$ for all $n$, then
\begin{align}
\label{limitfr}
\forall t\in [0,1] \,& \lim_{n\rightarrow \infty} \hat \fG_n (t) \,=\, \hat \fG (t) \,. 
\end{align}
Consider $t\in [0,1]$, let $x= \hat \fG (t)$ and $x_n= \hat \fG_n (t)$. Since $G_n \gestoch G$, we have $\fG_n \geq \fG$ thus $x_n \geq x$. Suppose that $\miniop{}{\limsup}{n\rightarrow +\infty} x_n := \overline{x} > x$. Up to extraction, we suppose that $\lim_{n\rightarrow +\infty}x_n = \overline{x}$. Choose $\beta \in (x, \overline{x})$ such that $\fG$ is continuous at $\beta$, thus $\lim_{n\rightarrow \infty} \fG_n (\beta) = \fG (\beta)$. On one hand, by the definition of $\hat \fG$ and the monotonicity of $\fG$, we have $\fG(\beta) <1- t$. On the other hand, $\beta< x_n$ for all $n$ large enough, thus $ \fG_n(\beta) \geq 1-t$ for all $n$ large enough, and we conclude that $\fG(\beta) = \lim_{n\rightarrow \infty} \fG_n(\beta)  \geq 1-t$, which is a contradiction, and~\eqref{limitfr} is proved.

$(ii)$ Similarly, if $G_n \lestoch G$ for all $n$, then
$ \displaystyle \;\forall t\in [0,1] \, \lim_{n\rightarrow \infty} \hat \fG_n (t) \,=\, \tilde \fG (t) \,. $

$(iii)$ We define $\underline{\fG}_n = \min \{\fG, \fG_n \}$ (resp. $\overline{\fG}_n = \max \{\fG, \fG_n\}$), and we denote by $\underline{G}_n$ (resp. $\overline{G}_n$) the corresponding probability measure on $[0,+\infty]$.
Notice that $\overline{G}_n \overset{d}{\rightarrow} G$ and $\underline{G}_n \overset{d}{\rightarrow} G$. Fix an edge $e$.
Then $\overline{G}_n \gestoch G$ for all $n$, and (i) implies that
$$\textrm{a.s.}  \quad  \lim_{n\rightarrow \infty} t_{\overline G_n} (e) \,=\, t_G (e) \,.$$
As $\underline G_n \lestoch G$ for all $n$ and  $t_{G} (e) = \tilde \fG (u(e))$ almost surely, (ii) implies that
$$\textrm{a.s.}   \quad  \lim_{n\rightarrow \infty} t_{\underline G_n} (e) \,=\, t_G (e) \,.$$
Finally, as $\underline G_n \lestoch G_n \lestoch \overline{G}_n$ for all $n$, we know by coupling that
$  t_{\underline G_n} (e) \le t_{G_n}(e) \le t_{\overline G_n} (e), $ which gives the desired convergence.
\end{proof}


\subsection{Stabilization of the point $\wx$ and monotonicity of the time constant}
\label{secstab}

We need to extend the monotonicity of the time constant to first-passage percolation on the infinite cluster of supercritical percolation. Since we work with different probability measures, the fact that, in the regularization process, the point $\wx^{\C_{G,M}}$ depends on $G$ may be disturbing. We get round this problem by considering an alternative probability measure $H$:
\begin{lemme}
\label{propstab}
Let $G$ and $H$ be probability measures on $[0,+\infty]$ such that $G \lestoch H$. For all $M\in \RRp$ satisfying $H([0,M]) > p_c(d)$, we have
$$ \mu_G(x) = \inf_{n\in \NN^*} \frac{\EE \left[\wT_G^{\C_{H,M}} (0,nx) \right] }{n} = \lim_{n\rightarrow \infty} \frac{\wT_G^{\C_{H,M}} (0,nx)}{n}  \; \textrm{a.s. and in } L^1. $$
\end{lemme}

\begin{proof}
Since $G \lestoch H$ we get by coupling that $t_G(e) \leq t_H (e)$ for all $e\in \EE^d$. Let $M\in \RRp$ satisfying $H([0,M])> p_c(d)$, then $G([0,M])> p_c(d)$ and $\C_{H,M} \subset \C_{G,M}$. The proof of the convergence of $\wT_G^{\C_{H,M}} (0,nx)/n$ is a straightforward adaptation of the proof of the convergence of $\wT_G^{\C_{G,M}} (0,nx)/n$: by the subadditive ergodic theorem, there exists a function $\mu'_{G,H} : \RR^d \rightarrow [0,+\infty)$ such that for all $x\in \ZZ^d$ we have
$$   \mu'_{G,H}(x) \,=\, \inf_{n\in \NN^*} \frac{\EE \left[\wT_G^{\C_{H,M}} (0,nx) \right] }{n} \,=\,\lim_{n\rightarrow \infty} \frac{\wT_G^{\C_{H,M}} (0,nx)}{n}\quad \textrm{a.s. and in } L^1\,.  $$
It remains to prove that $\mu'_{G,H} = \mu_G$.  For any $x\in \ZZ^d$, for any $\eps >0$, we have
\begin{align}
\label{eqstab}
\PP & \left[ \left| \wT_G^{\C_{H,M}} (0,nx)  - \wT_G^{\C_{G,M}} (0,nx)  \right| > n\eps \right] \nonumber \\
& \qquad \qquad \,\leq\, \PP \left[ T_G(\w0^{\C_{G,M}}, \w0^{\C_{H,M}} ) + T_G (\wnx^{\C_{G,M}}, \wnx^{\C_{H,M}}) > n\eps \right]\nonumber \\
& \qquad \qquad \,\leq\, 2 \, \PP \left[ T_G(\w0^{\C_{G,M}}, \w0^{\C_{H,M}} ) > n\eps/2 \right] \,.
\end{align}
Since $\w0^{\C_{G,M}} \in \C_{G,M} \subset \C_{G,\infty}$ and $\w0^{\C_{H,M}} \in \C_{H,M} \subset \C_{G,M} \subset \C_{G,\infty} $, the time $T_G(\w0^{\C_{G,M}}, \w0^{\C_{H,M}} )$ is finite a.s. thus the right hand side of inequality (\ref{eqstab}) goes to $0$ as $n$ goes to infinity. This concludes the proof of Lemma~\ref{propstab}.
\end{proof}

As a simple consequence of the coupling built in section \ref{seccoupling} and the stabilization  Lemma \ref{propstab}, we obtain the monotonicity of the function $G \mapsto \mu_G$:

\begin{lemme}
\label{propmon}
Let $G$, $H$ be probability measures on $[0,+\infty]$. we have
$$G\lestoch H \Longrightarrow \mu_G \leq \mu_H\,.$$
\end{lemme}

\begin{proof}
By construction of $\mu_G$ and $\mu_H$, it suffices to prove that $\mu_G (x) \leq \mu_H (x)$ for all $x\in \ZZ^d$. By coupling, since $ G \lestoch H$, we have $t_G(e) \leq t_H (e)$ for every edge $e$. Using Lemma~\ref{propstab} the conclusion is immediate, since we have a.s.
$$ \mu_G (x) \,=\, \lim_{n\rightarrow \infty}  \frac{\wT_G^{\C_{H,M}} (0,nx)}{n} \,\leq \,\lim_{n\rightarrow \infty}  \frac{\wT_H^{\C_{H,M}} (0,nx)}{n}  \,=\, \mu_H(x) \,.$$ 
\end{proof}


\subsection{Stabilization of the point $\wx$ for the Cheeger constant}
\label{secstab}

Concerning the Cheeger constant, we need a stabilization result similar to Lemma \ref{propstab}. For a path $r\in \mathcal{R} (x,y)$, let us define $\mathbf{b}_p (r) = \Card{\{ e \in \partial^+ r \,:\, e \text{ is } p-open\}}$. For $x,y\in \C_p$, we define $b_p (x,y) = \inf \{ \mathbf{b}_p (r) \,:\, r \in \mathcal{R} (x,y), r\textrm{ is }p-\text{open} \}$.

\begin{lemme}
\label{propstab2}
For any $p,p_0$ such that $p_c(2)< p_0 \leq p \leq 1$, for any $x \in \mathbb{R}^2$, we have
\begin{equation}\nonumber
        \beta_p(x)
        = \lim_{n \rightarrow \infty} \frac{b_p(\widetilde{0}^{\C_{p_0}}, \widetilde{nx}^{\C_{p_0}})}{n} \quad \quad \P_p-a.s. \text{ and in } L^1(\P_p).
\end{equation}
\end{lemme}

\begin{proof}
Exactly as in the proof of Lemma \ref{propstab}, since the convergence of $b_p(\widetilde{0}^{\C_{p}}, \widetilde{nx}^{\C_{p}})/n$ follows from a subadditive argument, the proof can be adapted straightforwardly to prove the existence of
$$  \beta_{p,p_0}(x)
        := \lim_{n \rightarrow \infty} \frac{b_p(\widetilde{0}^{\C_{p_0}}, \widetilde{nx}^{\C_{p_0}})}{n} \quad \quad \P_p-a.s. \text{ and in } L^1(\P_p). $$
The only thing we have to prove is the equality $\beta_{p,p_0}(x)=\beta_{p}(x)$. By the almost-subbadditivity of $b_p$ (see equation (2.27) in \cite{biskup2012isoperimetry}), we have
$$ b_p(\widetilde{0}^{\C_{p}}, \widetilde{nx}^{\C_{p}}) \,\leq b_p(\widetilde{0}^{\C_{p_0}}, \widetilde{nx}^{\C_{p_0}}) + b_p(\widetilde{0}^{\C_{p_0}}, \widetilde{0}^{\C_{p}}) + b_p(\widetilde{x}^{\C_{p_0}}, \widetilde{nx}^{\C_{p}}) +  4 \,.$$
Since $\C_{p_0} \subset \C_p$, there exists a finite simple (thus also right-most) path $\gamma$ which is $p$-open between $\widetilde{0}^{\C_{p_0}}$ and $\widetilde{0}^{\C_{p}}$, and by \cite[Lemma 2.5]{biskup2012isoperimetry} we know that $|\partial^+\gamma|<3|\gamma|$, thus $b_p(\widetilde{0}^{\C_{p_0}}, \widetilde{0}^{\C_{p}}) \leq 3|\gamma| < +\infty$. The same holds for $b_p(\widetilde{nx}^{\C_{p_0}}, \widetilde{nx}^{\C_{p}})$. As in the proof of Lemma \ref{propstab}, this is enough to conclude that $\beta_{p,p_0}(x)=\beta_{p}(x)$.
\end{proof}

Notice that Lemma \ref{propstab2} does not imply the monotonicity of the Cheeger constant. Indeed, consider $p_c(2) < p \leq q$, then
\begin{itemize}
\item a $q$-open path $\gamma$ may not be $p$-open,
\item a $p$-open path $\gamma$ is $q$-open by coupling, but $\mathbf{b}_q (\gamma)$ may be strictly bigger than $\mathbf{b}_p (\gamma)$,
\end{itemize}
thus no trivial comparison between $b_p(\widetilde{0}^{\C_{p}}, \widetilde{nx}^{\C_{p}})$ and $b_q(\widetilde{0}^{\C_{p}}, \widetilde{nx}^{\C_{p}})$ holds.


\section{Renormalization}
\label{secrenorm}

In this section we present the renormalization process and the construction of modified paths that will be useful to study both the time constant and the isoperimetric constant. We consider coupled bound i.i.d. Bernoulli percolations of different parameters. As we will see in Section \label{secetape2}, the construcion of modified paths in the model of first passage percolation reduces to this simplest case.


\subsection{Definition of the renormalization process}

Consider parameters $p_0$ and $q$ such that $p_c(d) < p_0  \leq q \leq 1$. Consider i.i.d. Bernoulii percolation on the edges of $\ZZ^d$ for these three parameters that are coupled, i.e. any $p_0$-open edge is $q$-open. Denote as before by $\mathcal C_{p_0}$ the a.s. unique infinite cluster of the supercritical Bernoulli field of parameter $p_0$. We call this field the $p_0$-percolation and its clusters the $p_0$-clusters.

We use a renormalization process in the spirit of the work of Antal and Pisztora~\cite{AP}. For a large integer $N$, that will be apropriately chosen later, we chop $\Zd$ into disjoint $N$-boxes as follows: we set $B_N$ to be the box $[-N,N]^d \cap \Zd$ and define the family of $N$-boxes by setting, for $\mathbf{i} \in \Zd$,
$$B_N(\mathbf{i})= \tau_{\mathbf{i}(2N+1)}(B_N),$$
where $\tau_b$ stands for the shift in $\Zd$ with vector $b \in \Zd$. We will also refer to the box $B_N(\mathbf{i})$ as the $N$-box with coordinate $\mathbf{i}$. The coordinates of $N$-boxes will be denoted in bold and considered as macroscopic sites, to distinghish them from the microscopic sites in the initial graph $\Zd$. We also introduce larger boxes:
for $\mathbf{i} \in \Zd$,
$$B'_N(\mathbf{i})= \tau_{\mathbf{i}(2N+1)}(B_{3N}).$$
As in \cite{AP}, we say that a connected cluster $C$ is a crossing cluster for a box $B$, if for all $d$ directions there is an open path contained in $C\cap B$ joining the the two opposite faces of the box $B$. 

Let $\C'_{p_0} = (\ZZ^d, \{ e\in \EE^d \,:\, e \textrm{ is }p_0\textrm{-open} \})$ be the graph whose edges are opened for the Bernoulli percolation of parameter $p_0$. We recall that $\C_{p_0}$ is the infinite cluster of $\C'_{p_0}$, and we have $D^{\C_{p_0}} (x,y) = D^{\C'_{p_0}} (x,y)$ for every vertices $x$ and $y$ in $\C_{p_0}$, and $D^{\C_{p_0}} (x,y)=+\infty$ if $x$ or $y$ are not in $\C_{p_0}$. Let us recall the following result, obtained by Antal and Pisztora~\cite[Theorem 1.1]{AP}, that says that the chemical distance $D^{\mathcal C'_{p_0}}$ can't be too large when compared to $\|\cdot \|_1$ or $\|\cdot \|_\infty$ (or any other equivalent norm): there exist positive constants $\hat A,\hat B, \beta$ such that
\begin{equation}
\label{EQ:AP}
\forall x \in \Zd \quad \P( \beta \|x\|_1 \le D^{\mathcal C'_{p_0}}(0,x)<+\infty) \le \hat A \exp(- \hat B\|x\|_1)\,,
\end{equation}
\begin{equation}
\label{EQ:APbis}
\forall x \in \Zd \quad \P( \beta \|x\|_{\infty} \le D^{\mathcal C'_{p_0}}(0,x)<+\infty) \le \hat A \exp(- \hat B\|x\|_{\infty})\,,
\end{equation}
and
\begin{equation}
\label{EQ:AP3}
\forall x \in \Zd \quad \P( \beta \|x\|_{2} \le D^{\mathcal C'_{p_0}}(0,x)<+\infty) \le \hat A \exp(- \hat B\|x\|_{2})\,.
\end{equation}
In fact Antal and Pisztora proved (\ref{EQ:AP}), but different norms being equivalent in $\RR^d$, we can obtain (\ref{EQ:APbis}) and (\ref{EQ:AP3}) by changing the constants.

\begin{defi}
\label{DEF:goodbox}
We say that the macroscopic site $\mathbf{i}$ is \emph{$(p_0,q)$-good} (or that the box $B_N(\mathbf{i})$ is $(p_0,q)$-good) if the following events occur:
\begin{itemize}
\item[(i)] There exists a unique $p_0$-cluster $\C$ in $B'_N(\mathbf{i})$ which has more than $N$ vertices;
\item[(ii)] This $p_0$-cluster $\C$ is crossing for each of the $3^d$ $N$-boxes included in $B'_N(\mathbf{i})$;
\item[(iii)] For all $x, y\in B_N'(\mathbf{i})$, if $\|x-y\|_\infty \geq N$ and $x$ and $y$ belong to this $p_0$-cluster $\C$, then $D^{\C'_{p_0}}(x,y)\leq  3 \beta N$;
\item[(iv)] If $\pi$ is a $q$-open path in $B'_N(\mathbf{i})$ such that $|\pi| \geq N$, then $\pi$ intersects this $p_0$-cluster $\C$ in $B'_N(\mathbf{i})$, i.e., they share a common vertex.
\end{itemize}
We call this cluster $\C$ the crossing $p_0$-cluster of the $(p_0,q)$-good box $B_N(\mathbf{i})$. 
\end{defi}
Otherwise, $B_N(\mathbf{i})$ is said to be \emph{$(p_0,q)$-bad}. For short, we say that $B$ is good or bad if there is no doubt about the choice of $(p_0,q)$.

On the macroscopic grid $\Zd$, we consider the same standard nearest neighbour graph structure as on the microscopic initial grid $\Zd$. Moreover we say that two macroscopic sites $\mathbf{i}$ and $\mathbf{j}$ are $*$-neighbors if and only if $\|\mathbf{i} - \mathbf{j}\|_{\infty} = 1$. If $C$ is a connected set of macroscopic sites, we define its exterior vertex boundary
$$\partial_v C=\{\mathbf{i} \in \Zd \backslash C: \; \mathbf{i} \text{ has a neighbour in } C \}.$$
For a bad macroscopic site $\mathbf{i}$, denote by $C(\mathbf{i})$ the connected cluster of bad macroscopic sites containing $\mathbf{i}$: the set $\partial_v C(\mathbf{i})$ is then a $*$-connected set of good macroscopic sites. For a good macroscopic site $\mathbf{i}$, we define $\partial_v C(\mathbf{i})$ to be $\{\mathbf{i}\}$.


\subsection{Modification of a path}

Let $p_c(d)<p_0 \leq p \leq q$ and $N$ be fixed. 
Let now $\gamma$ be a $q$-open path in $\Zd$. What we want to do is to remove from $\gamma$ the edges that are p-closed, and to look for bypasses for these edges using only edges that are $p_0$-open.

To $\gamma$, we associate the connected set $\Gamma \subset \Zd$ of $N$-boxes it visits: this is a lattice animal, \ie a connected finite set of $\Zd$, containing the box that contains the starting point of $\gamma$. We decompose $\gamma$ into two parts, namely $\gamma_a = \{e \in \gamma : e \textrm{ is } p\textrm{-open}\}$ and $\gamma_b = \{e \in \gamma : e \textrm{ is } p\textrm{-closed}\}$. We denote by $\Bad$ the (random) set of bad connected components of the macroscopic percolation given by the states of the $N$-boxes.

\begin{lemme}
\label{LEM:modif}
Assume that $y \in \mathcal C_{p_0}$, that $z \in \mathcal C_{p_0}$, that the $N$-boxes containing $y$ and $z$ are $(p_0,q)$-good and belong to an infinite cluster of $(p_0,q)$-good boxes. Let $\gamma$ be a $q$-open  
path between $y$ and $z$. Then there exists a $p$-open path $\gamma'$ between $y$ and $z$ that has the following properties :
\begin{enumerate}
\item $\gamma'\setminus \gamma$ is a collection of disjoint self avoiding $p_0$-open paths that intersect $\gamma' \cap \gamma$ only at their endpoints;
\item ${\displaystyle \Card{\gamma' \setminus \gamma} \leq \rho_d \left( N \sum_{C \in \Bad: \; C \cap \Gamma \ne \varnothing} \Card{C} + N^d \Card{\gamma_b} \right)}$, where $\rho_d$ is a constant depending only on the dimension $d$.
\end{enumerate}
\end{lemme}

Before proving Lemma \ref{LEM:modif}, we need a simpler estimate on the cardinality of a path inside a set of good blocks.

\begin{lemme}
\label{lemcard2}
There exists a constant $\hat \rho_d$, depending only on $d$, such that for every fixed $N$, for every $n\in \NN^*$, if $(B_N(\mathbf{i}))_{\mathbf{i}\in \mathcal{I}}$ is a $*$-connected set of $n$ $(p_0,q)$-good $N$-blocks, if $x\in \B_N(\mathbf{j})$ for $\mathbf{j}\in \mathcal I$ and 
$x$ is in the crossing $p_0$-cluster of $ \B_N(\mathbf{j})$, if $y\in \B_N(\mathbf{k})$ for $\mathbf{k}\in \mathcal I$ and 
$y$ is in the crossing $p_0$-cluster of $ \B_N(\mathbf{k})$, then there exists a $p_0$-open path from $x$ to $y$ of length at most equal to $\hat \rho_d (Nn + N^d)$.
\end{lemme}

\begin{proof}[Proof of Lemma \ref{lemcard2}]
Since $(B_N(\mathbf{i}))_{\mathbf{i}\in \mathcal{I}}$ is a $*$-connected set of good blocks, the definition of good boxes ensures that there exists a $p_0$-cluster $\C$ in $\C'_{p_0} \cap \cup_{\mathbf{i}\in \mathcal{I}} \{ e \in B'_N(\mathbf{i})\} $ 
which is crossing for every $N$-box included in $ \cup_{\mathbf{i}\in \mathcal{I}}  B'_{N}(\mathbf{i})$ (see Proposition 2.1 in Antal and Pisztora \cite{AP}). Since $x$ and $y$ are in $\C$,
there exists a path $\gamma = (\gamma_1,\dots , \gamma_p)$ from $x$ to $y$ in $\C'_{p_0} \cap \bigcup_{\mathbf{i}\in \mathcal{I}} \{ e \in B'_N(\mathbf{i}) \}$. Let $(\bm{ \phi}_i)_{1\leq i \leq r}$ be the path of macroscopic sites corresponding to the path of good blocks visited by $\gamma$ ($\bm{ \phi}$ may not be injective). Notice that $r\leq 3^d n $. We now extract a sequence of points along $\gamma$. Let $\Psi (1) =1$ and $j(1) = 1$. If $\Psi (1), \dots , \Psi (k) $ and $j(1), \dots , j(k)$ are defined, if the set $\{ i\geq \Psi (k) \,:\, \| \bm{ \phi}_i - \bm{ \phi}_{\Psi(k)} \|_\infty \geq 2  \}$ is non empty we define $\Psi(k+1) = \inf \{ i\geq \Psi (k) \,:\, \| \bm{ \phi}_i - \bm{ \phi}_{\Psi(k)} \|_\infty \geq 2  \}$ and we choose $j(k+1)\geq j(k)$ such that $\gamma_{j(k+1)} \in B_N(\bm{ \phi}_{\Psi(k+1)})$; if the set $\{ i\geq \Psi (k) \,:\, \| \bm{ \phi}_i - \bm{ \phi}_{\Psi(k)} \|_\infty \geq 2  \}$ is empty we stop the process. We obtain a sequence $(\gamma_{j(k)}, k=1 , \dots , r')$ of points, with $r' \leq r$.
By construction, for all $k\in \{1,\dots , r'-1  \}$, we have $\| \gamma_{j(k+1)}  - \gamma_{j(k)}  \|_\infty \geq N$ and
$$\| \bm{ \phi}_{\Psi(k+1)} - \bm{ \phi}_{\Psi(k+1)-1} \|_\infty \,=\, \| \bm{ \phi}_{\Psi(k+1)-1} - \bm{ \phi}_{\Psi(k)} \|_\infty \,=\, 1 \,,$$
thus $\gamma_{j(k)} \in B'_N(\bm{ \phi}_{\Psi(k+1) - 1})$ and $\gamma_{j(k+1)}\in B'_N(\bm{ \phi}_{\Psi(k+1) - 1})$.
For all $k\in \{1,\dots , r'-1 \}$, $B_N(\bm{ \phi}_{\Psi(k+1) - 1})$ is a good box, and $\gamma_{j(k)}$ and $\gamma_{j(k+1)}$ belong to the crossing $p_0$-cluster of $B_N(\bm{ \phi}_{\Psi(k+1) - 1})$, thus there exists a path from $\gamma_{j(k)}$ to $\gamma_{j(k+1)}$ in $\C'_{p_0}$ of length at most $3 \beta N$. By glueing these paths, we obtain a path from $x=\gamma_{j(1)}$ to $\gamma_{j(r')}$ in $\C'_{p_0}$ 
of length at most $3 \beta N r' \leq 3^{d+1} \beta N n$. Finally, since  
$y$ and $\gamma_{j(r')}$ belong to the crossing $p_0$-cluster of $B'_N(\bm{ \phi}_{\Psi(r') })$, 
there exists a path from $\gamma_{j(r')}$ to $y$ in $\C'_{p_0}$ of length at most $\Card{\{e \in B'_N(\bm{ \phi}_{\Psi(r') }) \}} \leq 2d \, 3^d N^d$.
\end{proof}

\begin{proof}[Proof of Lemma \ref{LEM:modif}]
To the path $\gamma$, we associate the sequence $\bm{ \phi}_0=(\bm{ \phi}_0(j))_{1 \le j \le r_0}$ of $N$-boxes it visits. Note that $\bm{ \phi}$ is not necessarily injective, and that the previously defined lattice animal $\Gamma$ is equal to $\bm{ \phi}_0(\{1, \dots,r_0\})$.

From the sequence $\bm{ \phi}_0$, we extract the subsequence $(\bm{ \phi}_1(j))_{1 \le j \le r_1}$, with $r_1 \le r_0$, of  $N$-boxes $B$ such that $\gamma \cap B$ contains at least one edge that is $p$-closed (more precisely, we keep the indices of the boxes $B$ that contain the smallest extremity, for the lexicographic order, of an edge of $\gamma$ that is $p$-closed). Notice that $r_1 \leq \Card{\gamma_b}$. The idea is the following:
\begin{itemize}
\item[$(1)$] If $\bm{ \phi}_1(j)$ is good, we add to $\gamma$ all the $p_0$-open edges in $B$: there will be enough such edges in the good $N$-box to find a by-pass for the edge of $\gamma$ that is $p$-closed.
\item[$(2)$] If $\bm{ \phi}_1(j)$ is bad, we will look for such a by-pass in the exterior vertex boundary $\partial_v C(\bm{ \phi}_1(j)))$ of the connected component of bad boxes of $\bm{ \phi}_1(j)$.
\end{itemize}
In the second case, we use Lemma \ref{lemcard2} to control the length of the by-pass we create. We recall that if $\mathbf{i}$ is good, then $\partial_v C( \mathbf{i}) = \{ \mathbf{i}\}$. Note that some $\partial_v C(\bm{ \phi}_1(j)))$ may coincide or be nested one in another or overlap. In order to define properly the modification of our path, we need thus to extract a subsequence once again, see Figure \ref{bypass1}.
\begin{figure}[h!]
\centering
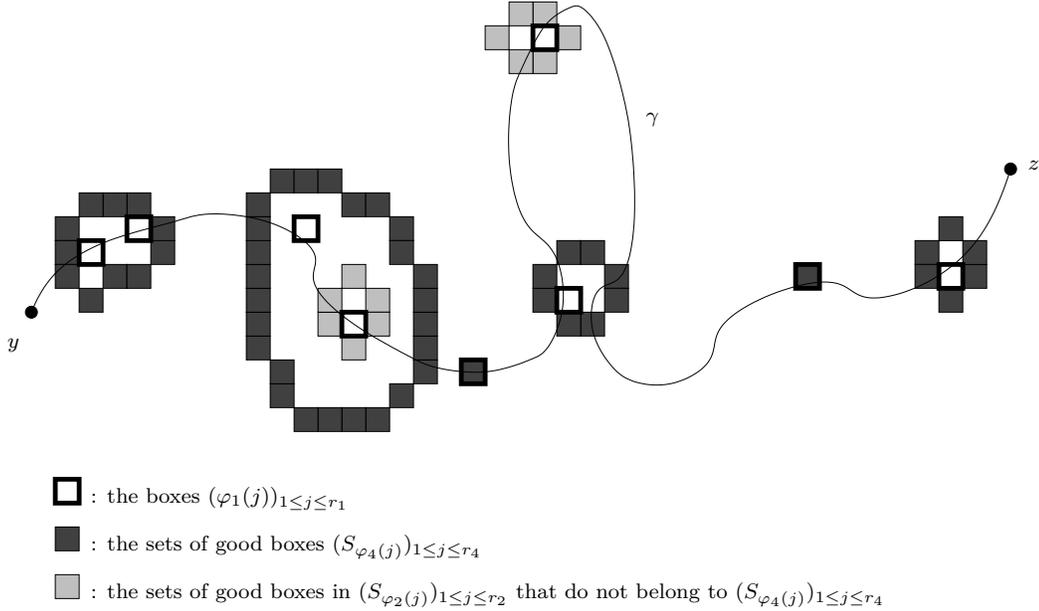
\caption{Construction of the path $\gamma'$ - step 1.}
\label{bypass1}
\end{figure}
We first consider the $*$-connected components $(S_{\varphi_2(j)})_{1 \le j \le r_2}$, with $r_2 \le r_1$, of the union of the $(\partial_v C(\bm{ \phi}_1(j)))_{1 \le j \le r_1}$, by keeping only the smallest index for each connected component. Next, in case of nesting, we only keep the largest connected component. We denote by $(S_{\varphi_3(j)})_{1 \le j \le r_3}$, with $r_3 \le r_2$, the remaining hypersurfaces of good $N$-boxes. Finally it may happen that $\gamma$ visits several times the same $S_{\varphi_3(j)}$ for some $j$: in this situation we can and must remove the loops that $\gamma$ makes between its different visits in $S_{\varphi_3(j)}$. Thus by a last extraction we obtain $(S_{\varphi_4(j)})_{1 \le j \le r_4}$, where $S_{\varphi_4(1)} = S_{\varphi_3(1)}$ and for all $k\geq 1$, $\varphi_4 (k+1)$ is the infimum of the indices $(\varphi_3(j))_{1\le j \le r_3}$ such that $\gamma$ visits $S_{\varphi_3(j)}$ after it exits $S_{\varphi_4(k)} $ for the last time (if such a $j$ exists).

Note that the path $\gamma$ must visit each $(S_{\varphi_4(j)})_{1 \le j \le r_4}$. We now cut $\gamma$ in several pieces. Let $\Psi_{in}(1) = \min \{ k\geq 1 \,:\, \gamma_k  \in \cup_{\mathbf{i} \in S_{\varphi_4(1)} } B_N(\mathbf{i})\} $ and $\Psi_{out}(1) = \max \{ k\geq \Psi_{in} (1) \,:\, \gamma_k\in \cup_{\mathbf{i} \in S_{\varphi_4(1)} } B_N(\mathbf{i}) \} $ (see Figure \ref{bypass2}).
\begin{figure}[h!]
\centering
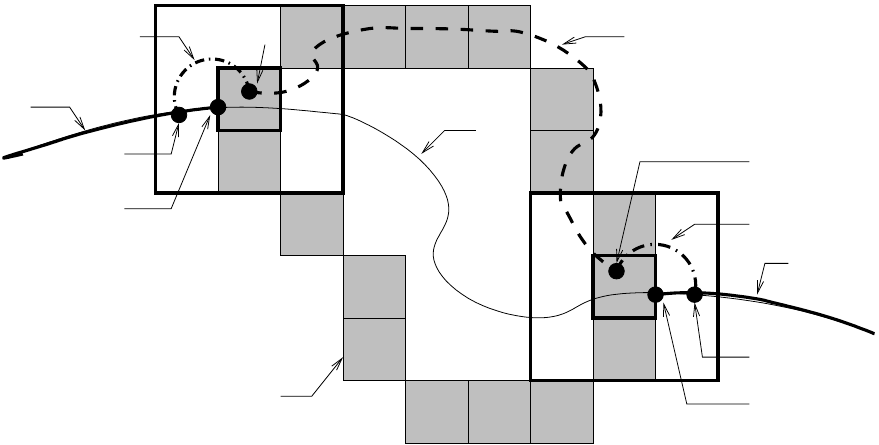
\caption{Construction of the path $\gamma'$ - step 2.}
\label{bypass2}
\end{figure}
By recurrence, for all $1 \le j \le r_4$, we define $\Psi_{in}(j) = \min \{ k\geq \Psi_{out} (j-1) \,:\, \gamma_k  \in \cup_{\mathbf{i} \in S_{\varphi_4(j)} } B_N(\mathbf{i})\} $ and $\Psi_{out}(j) = \max \{ k\geq \Psi_{in} (j) \,:\, \gamma_k  \in \cup_{\mathbf{i} \in S_{\varphi_4(j)} } B_N(\mathbf{i})\} $. For all $1 \le j \le r_4-1$, let $\gamma_j$ be the part of $\gamma$ from $\gamma_{\Psi_{out}(j)}$ to $\gamma_{\Psi_{in}(j+1)}$. By construction $\gamma_j$ contains no $p$-closed edge, and has at least $N$ vertices in $B'_{N}(\mathbf{i})$ for some $\mathbf{i} \in S_{\varphi_4(j)}$ (resp. in $B'_{N}(\mathbf{k})$ for some $\mathbf{k} \in S_{\varphi_4(j+1)}$), thus $\gamma_j$ intersects the crossing $p_0$-cluster of $B'_{N}(\mathbf{i})$ (resp. $ B'_{N}(\mathbf{k})$) ; let us denote by $x_j$ (resp. $y_{j+1}$) the last (resp. first) intersection of $\gamma_j$ with the $p_0$-cluster of $B'_{N}(\mathbf{i})$ (resp. $ B'_{N}(\mathbf{k})$). The vertex $x_j$ (resp. $y_{j+1}$) is not inside $B_{N}(\mathbf{i})$ (resp. $ B_{N}(\mathbf{k})$), but it is connected inside the $p_0$-cluster of $B'_{N}(\mathbf{i})$ (resp. $ B'_{N}(\mathbf{k})$) to a vertex $x'_j$ (resp. $y'_{j+1}$) of $B_{N}(\mathbf{i})$ (resp. $ B_{N}(\mathbf{k})$) by a path $\gamma'_{j,out}$ (resp. $\gamma'_{j+1,in}$) of length at most equal to $2 d 3^d N^d \leq \hat \rho_d N^d$. Let us study more carefully the beginning of the path $\gamma$. Since the $N$-box containing $y$ belongs to an infinite cluster of good boxes, it cannot be in the interior of $S_{\varphi_4(1)}$. If the box containing $y$ is not in $S_{\varphi_4(1)}$ (thus it is outside $S_{\varphi_4(1)}$), denote by $\gamma_0$ the portion of $\gamma$ from $0$ to $\gamma_{\Psi_{in}(1)}$, and define as previously $y_1, y_1'$ and $\gamma'_{1,in}$. If the box containing $y$ is in $S_{\varphi_4(1)}$, then $\Psi_{in}(1) = 1$ and $\gamma_{\Psi_{in}(1)} =y$. As $y \in \mathcal C_{p_0}$ and the box containing $y$ is good, $y$ is in the crossing $p_0$-cluster of the box containing $y$, thus we can define $y_1 = y_1' = y$, $\gamma_0 = \emptyset$ and $\gamma'_{1,in}=\emptyset$. Similarly, we define $x_{r_4}, x'_{r_4}, \gamma_{r_4}$ and $\gamma'_{r_4, out}$ depending on the fact that the box containing $z$ belongs to $S_{\varphi_4(r_4)}$ or not. For all $1 \le j \le r_4$, we can apply Lemma \ref{lemcard2} to state that there exists a $p_0$-open path $\gamma'_{j,link}$ from $y'_j$ to $x'_{j}$ of length at most $\hat \rho_d (N^d + N \Card{S_{\varphi_4(j)}})$.

For all $1 \le j \le r_4$, define $\gamma'_j = \gamma'_{j,in} \cup \gamma'_{j,link} \cup \gamma'_{j,out}$. By construction each $\gamma'_j$ is $p_0$-open. We can glue together the paths $\gamma_0, \gamma'_1, \gamma_1, \gamma'_2, \dots, \gamma'_{r_4}, \gamma_{r_4}$ in this order to obtain a $p$-open path $\gamma'$ from $y$ to $z$. Up to cutting parts of these paths, we can suppose that each $\gamma_i'$ is a self-avoiding path, that the $\gamma_i'$ are disjoint and that each $\gamma_i'$ intersects only $\gamma_{i-1}$ and $\gamma_{i}$, and only with its endpoints.

Finally we need an estimate on $\Card {\gamma' \smallsetminus \gamma}$. Obviously $\gamma' \smallsetminus \gamma \subset \cup_{i=1}^{r_4} \gamma_i'$, thus
\begin{align*}
\Card {\gamma' \smallsetminus \gamma} & \,\leq\, 2 r_4 \hat \rho_d N^d + \sum_{i=1}^{r_4} \hat \rho_d (N^d + N \Card{S_{\varphi_4(j)}})\\
& \,\leq \, 3 r_4  \hat \rho_d N^d  +  \hat \rho_d N  \sum_{i=1}^{r_4}\Card{S_{\varphi_4(j)}} \\
& \,\leq \,3  \hat \rho_d N^d \Card{\gamma_b} + \hat \rho_d N  \sum_{C \in \Bad: \; C \cap \Gamma \ne \varnothing} \Card{\partial_v C} \,.
\end{align*}
To conclude, we just have to remark that $ \Card{\partial_v C} \le 2d  \Card{C}$.
\end{proof}


\subsection{Probabilistic estimates}

We want to bound the probability that $\Card{\gamma' \smallsetminus \gamma}$ is big for $q-p_0$ small enough. Lemma \ref{LEM:modif} makes appear the connected set $\Gamma \subset \Zd$ of $N$-boxes visited by the path $\gamma$. To control $\Card{\gamma'\smallsetminus \gamma}$, we need to have a deterministic control on $\Card{\Gamma}$. This is the purpose of the following Lemma.

\begin{lemme}
\label{lemcard3} 
There exists a constant $\tilde C_d$, depending only on $d$, such that for every path $\gamma$ of $\ZZ^d$, for every $N\in \NN^*$, if $\Gamma$ is the animal of N-blocks that $\gamma$ visits, then  
$$\Card{\Gamma}  \,\leq\,  \tilde C_d\left( 1+ \frac{\Card{\gamma}+1 }{N} \right)  -1 \,.$$
\end{lemme}

\begin{proof}
Let $\gamma = (\gamma_i)_{i=1,\dots, n}$ be a path of $\ZZ^d$ for a $n\in \NN^*$ ($\gamma_i$ is the $i$-th vertex of $\gamma$, $n= |\gamma| +1$), and fix $N\in \NN^*$. Let $\Gamma$ be the animal of N-blocks that $\gamma$ visits. We will include $\Gamma$ in a bigger set of blocks whose size can be controlled. Let $p(1) = 1$ and $\mathbf{i}_1$ be the macroscopic site such that $\gamma_1 \in B_N(\mathbf{i}_1)$. If $p(1),\dots , p(k)$ and $\mathbf{i}_1, \dots, \mathbf{i}_k$ are constructed, define $p(k+1) = \inf \{ j \in \{p(k),\dots, n\} \,:\, \gamma_j \notin B_N'(\mathbf{i}_k)  \}$ if this set is not empty and let $\mathbf{i}_{k+1}$ be the macroscopic site such that $\gamma_{p (k+1)} \in B_N(\mathbf{i}_{k+1})$, and stop the process if for every $j \in \{p(k),\dots, n\}$ , $\gamma_j \in B_N'(\mathbf{i}_k)$. We obtain two finite sequences $(p(1), \dots, p(r)) $ and $(\mathbf{i}_1, \dots , \mathbf{i}_r)$. First notice that
$$ \Gamma \,\subset\, \bigcup_{k=1}^r B_N'(\mathbf{i}_k) $$
by construction, thus $\Card{\Gamma} \leq 3^d r -1 $. Moreover for every $k\in \{1,\dots , r-1\}$, $\|\gamma_{p(k+1)} - \gamma_{p(k)} \|_1 \geq N$, thus $p(k_1) - p(k) \geq N$. This implies that $N (r-1) \leq p(r) - p(1) \leq n$, and we conclude that
$$  \Card{\Gamma} \,\leq \, 3^d \left( 1+ \frac{n}{N} \right)   -1 \,.$$
\end{proof}

Then we need a control on the probability that a block is good.

\begin{lemme}
\label{couplage}
\begin{itemize}
\item[(i)] For every $q> p_c (d)$, there exists $\delta_0 (q)>0$ such that if $p_0 \in (p_c(d) ,q ]$ satisfy $q-p_0 \leq \delta_0$, then for every $\mathfrak p<1$, there exists an integer $N(p_0, q, \mathfrak p)$ such that the field $(\1_{\{B_N(\mathbf{i}) \text{ is $(p_0,q)$-good}\}})_{\mathbf{i} \in \Zd}$ stochastically dominates a family of independent Bernoulli random variables with parameter $\mathfrak p$. 
\item[(ii)] For every $p_0 >p_c (d) $, there exists $\delta_1 (p_0)>0$ such that if $ q_1 \in [p_0 ,1] $ satisfy $q_1-p_0 \leq \delta_1$, then for every $\mathfrak p<1$ there exists an integer $N'(p_0, q_1, \mathfrak p)$ such that for any $q\in [p_0,q_1]$ the field $(\1_{\{B_N'(\mathbf{i}) \text{ is $(p_0,q)$-good}\}})_{\mathbf{i} \in \Zd}$ stochastically dominates a family of independent Bernoulli random variables with parameter $\mathfrak p$.
\end{itemize}
\end{lemme}
\begin{proof}
Obviously, the states of $(B_N(\mathbf{i}))_{\mathbf{i} \in \Zd}$ 
have a finite range of dependance 
and are identically distributed. Then, by the Liggett--Schonmann--Stacey Theorem~\cite{LSS}, it is sufficient to check that $\lim_{N \to +\infty} \P(B_N \text{ good})=1$, for $B_N = B_N(\mathbf{0})$.

Consider first the properties (i) and (ii) of the Definition \ref{DEF:goodbox}, that depend only on $p_0$. For any $p_0 > p_c(d)$, when $d\ge 3$, the fact that 
$$\lim_{N \to +\infty} \P(B_N \text{ satisfies (i) and (ii)})=1$$
follows from the Pisztora coarse graining argument (see Pisztora~\cite{MR97d:82016} or the coarse graining section in Cerf~\cite{MR2241754}), see also for instance Grimmett \cite{grimmett-book} Lemma (7.104). When $d=2$, see Couronn{\'e} and Messikh~\cite{MR2078538}. We now study the property (iii) in the Definition \ref{DEF:goodbox}, that also depends only on $p_0$. Using Antal and Pisztora's estimate (\ref{EQ:AP}), for any fixed $p_0>p_c(d)$, we have for all $N$
\begin{align*}
\PP [B_N & \text{ does not satisfy (iii)} ]\\
&  \,\leq\, \sum_{x\in B_N'} \sum_{y\in B_N'} \mathbbm{1}_{\|x-y\|_\infty \geq N} \PP \left[ x\overset{\C'_{p_0}}{\longleftrightarrow} y \,,\, D^{\C'_{p_0}} (x,y) \geq 3 \beta N\right]\\
& \,\leq \,  \sum_{x\in B_N'} \sum_{y\in B_N'} \mathbbm{1}_{\|x-y\|_\infty \geq N} \PP \left[ x\overset{\C'_{p_0}}{\longleftrightarrow} y \,,\, D^{\C'_{p_0}} (x,y) \geq \beta \|x-y\|_\infty \right]\\
& \,\leq \,  \sum_{x\in B_N'} \sum_{y\in B_N'} \mathbbm{1}_{\|x-y\|_\infty \geq N} \hat A e^{-\hat B \|x-y\|_\infty}
 \,\leq \, (3N)^d . (3N)^d \hat A e^{-\hat B N}\,
\end{align*}
that goes to $0$ when $N$ goes to infinity. The delicate part of the proof is the study of the property (iv) in the Definition \ref{DEF:goodbox}. For $q=p_0$, we are done since property (iv) is implied by the uniqueness of the $p_0$-crossing cluster in $B_N'$. We want to deduce from this that property (iv) is asymptotically typical. We follow the proof of Russo's formula, see for instance Theorem 2.25 in \cite{grimmett-book}. For given parameters $p_c(d)<p_0 < p \leq 1$, we denote by $\PP_{p_0, p}$ the probability of the corresponding coupled Bernoulli percolation, and we declare that
\begin{itemize}
\item an edge $e$ is in state $0$ if $e$ is $p$-closed,
\item an edge $e$ is in state $1$ if $e$ is $p_0$-closed and $p$-open,
\item an edge $e$ is in state $2$ if $e$ is $p_0$-open.
\end{itemize}
We define $A_N$ as the event that there exists a crossing cluster $\C$ of edges of state $2$ in $B'_N$, and a path $\pi \subset B_N'$ of edges of state $1$ or $2$ such that $\Card{\gamma} = N$ and $\gamma$ does not intersect $\C$. Let us fix $p_0$. When $p$ vary, the edges of state $2$ remain unchanged, we only change the state of edges from $0$ to $1$ and conversely. For a given $p_0$, the event $A_N$ is increasing in $p$. We denote by $\mathcal N (A_N)$ the random number of edges that are $0-1$-pivotal for $A$, i.e., the number of edges $e$ such that if $e$ is in state $1$ then $A_N$ occurs, and if $e$ is in state $0$ then $A_N$ does not occur. Following the proof of Russo's formula, we obtain that
$$ \frac{\partial}{\partial p} \PP_{p_0,p}  (A_N) \,=\, \frac{1}{p} \EE_{p_0, p} [ \mathcal N (A_N) \, |\, A_N] \, \PP_{p_0,p} (A_N) \,. $$
We remark that when $A$ occurs, $\mathcal N ( A_N) \leq N$, the length of the desired path, thus
$$  \EE_{p_0, p} [ \mathcal N (A_N) \, |\, A_N] =  \EE_{p_0, p} [ \mathbbm{1}_{A_N} \mathcal N (A_N) \, |\, A_N]  \leq N \,.$$
We obtain that
\begin{align}
\label{eqRusso1}
\PP_{p_0,q} (A_N) & \,=\, \PP_{p_0,p_0} (A_N) \exp \left( \int_{p_0}^q  \frac{1}{p} \EE_{p_0, p} [ \mathcal N (A_N) \, |\, A_N] \, dp \right) \nonumber\\
& \,\leq\,  \PP_{p_0,p_0} (A_N) \exp \left( N \int_{p_0}^q  \frac{1}{p}  \, dp \right) \nonumber \\
& \,\leq \,  \PP_{p_0,p_0} (A_N) \exp \left( N \log \left( \frac{q}{p_0} \right) \right) \nonumber \\
& \,\leq \, \PP_{p_0,p_0} (A_N) \exp \left( N \log \left( 1+ \frac{q-p_0}{p_0} \right) \right)
\end{align}
It comes from the coarse graining arguments previously cited to study property (i) that $\PP_{p_0,p_0} (A_N)$ decays exponentially fast with $N$: there exists $\kappa_1(p_0), \kappa_2(p_0)$ such that 
\begin{equation}
\label{eqRusso2}
 \PP_{p_0,p_0} (A_N) \,\leq \, \kappa_1(p_0) e^{-\kappa_2(p_0) N} \,.
\end{equation}

\noindent
\underline{\em Part(ii) of Lemma \ref{couplage}:}
When $p_0$ is fixed, combining \eqref{eqRusso1} and \eqref{eqRusso2} is enough to conclude that there exists $\delta_1 (p_0) >0$ such that if $q_1<p_0 + \delta_1$, then 
\begin{equation}
\label{eqRusso4}
\lim_{N\rightarrow \infty} \PP_{p_0,q_1} (A_N) \,=\, 0 \,.
\end{equation}
We conclude that for every $p_0 >p_c (d) $, there exists $\delta_1 (p_0)>0$ such that if $ q_1 \in [p_0 ,1] $ satisfy $q_1-p_0 \leq \delta_1$, then for every $\mathfrak p<1$ there exists an integer $N'(p_0, q_1, \mathfrak p)$ such that for $q=q_1$ the field $(\1_{\{B_N'(\mathbf{i}) \text{ is $(p_0,q_1)$-good}\}})_{\mathbf{i} \in \Zd}$ stochastically dominates a family of independent Bernoulli random variables with parameter $\mathfrak p$. The only property of a good block that depends on $q$ is property (iv), and if $p_0 \leq q \leq q_1$ then any $q$-open path is also a $q_1$-open path, thus if a block is $(p_0,q_1)$-good then it is $(p_0,q)$-good for any parameter $q\in [p_0, q_1]$ . We conclude that for $N' = N'(p_0, q_1, \mathfrak p)$, for any $q\in [p_0,q_1]$, the field $(\1_{\{B_N'(\mathbf{i}) \text{ is $(p_0,q)$-good}\}})_{\mathbf{i} \in \Zd}$ stochastically dominates a family of independent Bernoulli random variables with parameter $\mathfrak p$.\\

\noindent
\underline{\em Part(i) of Lemma \ref{couplage}:}
If $q$ is fixed, we need to replace \eqref{eqRusso2} by a control on $ \PP_{p_0,p_0} (A_N)$ which is uniform for $p_0$ in a left neighborhood of $q$. Let us have a closer look at the proof of \eqref{eqRusso2}. In dimension $d\geq3$, we refer to the proof of Lemma 7.104 in Grimmett \cite{grimmett-book} : the constants $\kappa_1(p_0),\kappa_2(p_0)$ of \eqref{eqRusso2} appearing in Grimmett's book are explicit functions of the parameters $\delta (p_0)$ and $L(p_0)$ chosen in Lemma 7.78 in \cite{grimmett-book}. The probability controlled in Lemma 7.78 in \cite{grimmett-book} is clearly non decreasing in the parameter $p$ of the percolation, thus the choice of $\delta (p)$ and $L(p)$ made for a given $p>p_c(d)$ can be kept unchanged for any $p' \geq p$. Fixing $p_0' = (q-p_c(d))/2$, we obtain that for any $p_0 \in [p_0', q]$,
\begin{equation}
\label{eqRusso3}
 \PP_{p_0,p_0} (A_N) \,\leq \, \kappa_1(p'_0) e^{-\kappa_2(p'_0) N} \,.
\end{equation}
Combining \eqref{eqRusso1} and \eqref{eqRusso3} we can conclude that in dimension $d\geq3$, when $q$ is fixed, there exists $\delta_0 (q)$ such that if $p_0 \in [ p'_0 , q]$ satisfies $q-p_0 \leq \delta_0$, then \eqref{eqRusso4} still holds. In dimension $2$, \eqref{eqRusso2} is obtained by Couronn{\'e} and Messikh~\cite{MR2078538}, Theorem 9, in  a more general setting. The constants appearing in this theorem are explicit functions of the constants appearing in Proposition 6 in \cite{MR2078538}, and the same remark as in dimension $d\geq3$ leads to the uniform control \eqref{eqRusso3}, and the proof is complete.
\end{proof}

We can now use Lemma \ref{couplage} to bound the probability that $ \sum_{C \in \Bad: \; C \cap \Gamma \ne \varnothing} \Card{ C}$ is big. Denote by $\Animals$ the set of lattice animals containing $\mathbf{0}$, and $\Animals_n$ the subset of those having size $n$.

\begin{lemme}
\label{lem:new}
Let $\epsilon >0$. Let $p_c^{\textrm{site}}(d)$ be the critical parameter for independent Bernoulli site percolation on $\mathbb{Z}^d$. Choose $\alpha = \alpha(\epsilon)>0$ and then $\mathfrak p = \mathfrak p(\epsilon) \in (p_c^{\textrm{site}}(d),1)$, such that
\begin{align}
7^d \exp(-\alpha \varepsilon) & \le\frac{1}{3}, \label{choixalpha} \\
\mathfrak p+\frac{e^\alpha 7^d (1-\mathfrak p)}{1-e^\alpha 7^d (1-\mathfrak p) } &\le \frac{3}{2}. \label{choixp}
\end{align}
For a given $q>p_c(d)$ (resp. $p_0 > p_c(d)$), for a fixed $p_0 \in (p_c(d), q]$ such that $q-p_0 \leq \delta_0 (q)$ (resp. $q_1\geq p_0$ such that $q_1-p_0 \leq \delta_1 (p_0)$ and any $q\in [p_0,q_1]$),
let finally $N = N(p_0, q, \mathfrak p (\epsilon))$ (resp. $N = N'(p_0, q_1, \mathfrak p (\epsilon))$) be large enough to have the stochastic comparison of Lemma~\ref{couplage} with this parameter $\mathfrak p (\epsilon)$. Then for all $m\in \NN$, we have
$$  \P \left( \exists \Gamma \in \Animals, \;  
|\Gamma| \ge \frac{m}{N} , \; \sum_{ C \in \Bad: \; C \cap \Gamma \ne \varnothing}  |C| \ge \varepsilon \Card{\Gamma}
\right) \,\leq\, e^{-\frac{m}{N} +1} \,. $$
\end{lemme}

\begin{proof}
We have
\begin{align*}
 \mathcal P (m) &   \stackrel{\textrm{def}}{=} \P   \left( \exists \Gamma \in \Animals, \;  
|\Gamma| \ge \frac{m}{N} , \; \sum_{ C \in \Bad: \; C \cap \Gamma \ne \varnothing}  |C| \ge \varepsilon \Card{\Gamma}
\right)\\
 & \leq   \sum_{n \ge \frac{m}{N}} \sum_{\Gamma \in \Animals_n} \mathbb P \left( \sum_{C  \in \Bad: \; C \cap \Gamma \ne \varnothing}  |C| \ge  \varepsilon |\Gamma| \right) \\
 & \leq  \sum_{n \ge \frac{m}{N}} \sum_{\Gamma \in \Animals_n} \mathbb P_{\mathfrak p} \left( \sum_{C  \in \Bad: \; C \cap \Gamma \ne \varnothing}  |C| \ge \varepsilon |\Gamma| \right).
\end{align*}
For the last inequality, we use the coupling Lemma~\ref{couplage} to replace the locally dependent states of our $N$-boxes by an independent Bernoulli site percolation with parameter $\mathfrak p$ chosen in~\eqref{choixp}. From now on, we work with this  Bernoulli site percolation with parameter $\mathfrak p$.
Denote by $C(0)$ the connected component of closed sites containing $0$ (with the convention that if $0$ is open, then $C(0)=\varnothing$). 
Let $(\tilde C(i))_{i \in \Zd}$ be independent and identically distributed random sets of $\Zd$ with the same law as $C(0)$. 
Fix a set $\Gamma=(\Gamma(i))_{1 \le i \le n}$ of sites; we first prove that, for the independent Bernoulli site percolation, the following stochastic comparison holds: 
\begin{equation}
\label{EQ:stochcomp}
 \sum_{C \in \Bad: \; C \cap \Gamma \ne \varnothing} |C| \lestoch \sum_{i=1}^n |\tilde C(i)|.
 \end{equation}
The idea is to build algorithmically the real clusters from the sequence of pre-clusters $(\tilde C(i))_{i \in \Zd}$, as in the work of Fontes and Newman \cite{MR1233623}, proof of Theorem~4. Note however that in our sum \eqref{EQ:stochcomp}, each visited cluster is only counted once, while they count each cluster the number of times it is visited, which explains the difference between our stochatic domination and their one.  
We proceed by induction on $j \in \{1, \dots,n\}$ to build a new family $(\overline C(i))_{1 \le i \le n }$ such that 
$$A_j \stackrel{def}{=}\bigcup_{C \in \Bad: \; C \cap \{\Gamma(i):\; 1\le i \le j\} \ne \varnothing} C \stackrel{law}{\subset}\bigcup_{i=1}^j \overline C(i)\subset \bigcup_{i=1}^j ( \Gamma(i)+ \tilde C(i)).
$$
 Set $\overline C(1)=\Gamma(1)+\tilde C(1)$.  Assume now that $(\overline C(i))_{1 \le i \le j}$ are built for some $j<n$:
\begin{itemize}
\item if $\Gamma(j+1) \in A_j$, then 
$A_{j+1}=A_j$,
so we set $\overline C(j+1)=\varnothing$; 
\item if $\Gamma(j+1) \in \partial_v A_j$ (the exterior vertex boundary of $ A_j$), then it is a good site, so we set $\overline C(j+1)=\varnothing$;
\item otherwise, the conditional distribution of the bad cluster $C$ containing the site $\Gamma(j+1)$, given $A_j$, is that of the percolation cluster of $\Gamma(j+1)$ in a site percolation model where $\Zd$ is replaced by $\Zd \backslash (A_j\cup \partial_v A_j)$; thus, it has the same law as the connected component of $\Gamma(j+1)$ in  
$$\overline C(j+1){=} \left( \Gamma(j+1) + \tilde C(j+1) \right)   \backslash 
\left( A_j \cup \partial_v A_j \right) ,$$ 
\end{itemize}
which ends the construction and proves \eqref{EQ:stochcomp}. As the number of lattice animals containing $0$ with size $n$ is bounded from above by $(7^d)^n$ (see Kesten~\cite{MR692943}, p 82. or Grimmett~\cite{grimmett-book}, p.85), we have, by the Markov inequality, 
\begin{align*}
 \mathcal P (m) \le & \sum_{n \ge \frac{m}{N}} (7^d)^n \exp(-\alpha \varepsilon n) \left( \mathbb E_{\mathfrak p}(\exp(\alpha | C(0)|)) \right)^n .
\end{align*}
But
\begin{align*}
\mathbb E_{\mathfrak p}(\exp(\alpha | C(0)|)) & = \mathfrak p+ \sum_{k\ge 1} \exp(\alpha k) \P_{\mathfrak p}( | C(0)|=k) 
 \le  \mathfrak p+ \sum_{k\ge 1} \exp(\alpha k) \P_{\mathfrak p}( |C(0)| \ge k) \\
& \le  \mathfrak p+ \sum_{k\ge 1} \exp(\alpha k)(7^d)^k(1- \mathfrak p)^k = \mathfrak p+\frac{e^\alpha 7^d (1- \mathfrak p)}{1-e^\alpha 7^d (1- \mathfrak p) }.
\end{align*}
With the choices \eqref{choixalpha} and \eqref{choixp} we made for $\alpha$ and $\mathfrak p$, this ensures that
\begin{equation*}
\mathcal P (m) \le \sum_{n \ge \frac{m}{N}} 2^{-n} \le 2^{-\frac{m}{N}+1}.
\end{equation*} 
\end{proof}


\section{Truncated passage times, proof of Theorems~\ref{propetape2}}
\label{secetape2}

Let $G$ be a probability measure on $[0,+\infty]$ such that $q := G([0,+\infty)) > p_c(d)$. Let $\delta_0 (q)$ be given by Lemma \ref{couplage}. Fix $M_0$ large enough so that $p_0:= G([0,M_0]) > p_c(d)$ and $q-p_0 \leq \delta_0$. For a $K\in [M_0, +\infty)$, define $p=p(k) = G([0,K])$. We define the following bound i.i.d. Bernoulli percolations :
\begin{itemize}
\item an edge $e$ is declared $p_0$-open if and only if $t_G(e) \leq M_0$,
\item an edge $e$ is declared $p$-open if and only if $t_G(e) \leq K$,
\item an edge $e$ is declared $q$-open if and only if $t_G(e)<\infty$.
\end{itemize}
These percolations are naturally coupled, thus we can use the modification of paths presented in the previous section. Denote as before by $\mathcal C_{G,M_0}$ the a.s. unique infinite cluster of the supercritical Bernoulli field $\{\ind{t_G(e) \le M_0}: \; e \in \Ed\}$. We call this field the $M_0$-percolation and its clusters the $M_0$-clusters. They correspond exactly to the $p_0$-percolation and the $p_0$-clusters.


\subsection{Estimation for the passage time of the modified path}

\begin{lemme}
\label{LEM:modifdeux}
There exists a positive constant $\rho'_d$ (depending only on $d$ and $M_0$) such that the following holds: 
Assume that $y \in \mathcal C_{G,M_0}$, that $z \in \mathcal C_{G,M_0}$, that the $N$-boxes containing $y$ and $z$ are good and belong to an infinite cluster of good boxes. Then for every $K\geq M_0$,
$$ T_{G} (y,z) \le T_{G^K} (y,z)\left( 1+\frac{\rho'_d  N^d}{K}\right) +\rho'_d N \, \sum_{C \in \Bad: \; C \cap \Gamma \ne \varnothing} \Card{C},$$
where $\Gamma$ is the lattice animal of $N$-boxes visited by an optimal path between $y$ and $z$ for the passage times with distribution $G^K$.
\end{lemme}

\begin{proof} 
As $y \in \mathcal C_{G,M_0}$ and $z \in \mathcal C_{G,M_0}$, the quantities $T_{G}(y,z)$ and  $T_{G^K} (y,z)$ are bounded by $M_0$ times the chemical distance in $\mathcal C_{G,M_0}$ between $y$ and $z$, and are thus finite. 
Let $\gamma$ be an optimal path between $y$ and $z$ for $T_{G^K} (y,z)$ : since its passage time is finite, $\gamma$ is $q$-open, and we can consider the modification $\gamma'$ given by Lemma \ref{LEM:modif}. Since $\gamma'$ is a path between $y$ and $z$, and $\gamma' \smallsetminus \gamma $ is $p_0$-open, we have
\begin{align*}
T_{G} (y,z)  &  \le \sum_{e \in \gamma'} t_G(e)  = \sum_{e \in \gamma \cap \gamma'}  t_G(e)  + \sum_{e \in \gamma' \setminus \gamma}  t_G(e) 
 \le \sum_{e \in \gamma_a}  t_G(e) + M_0\,  \Card{\gamma'\setminus \gamma} \,.
 \end{align*}
On one hand, since $\gamma$ is an optimal path between $y$ and $z$ for $T_{G^K} (y,z)$, we have
$$ \sum_{e \in \gamma_a}  t_G(e) = \sum_{e \in \gamma_a}  t_{G^K}(e) \leq \sum_{e \in \gamma}  t_{G^K}(e) =T_{G^K} (y,z) \,. $$
On the other hand, using the estimate on the cardinality of $\gamma'\setminus \gamma$ given in Lemma~\ref{LEM:modif}, and noticing that the number of edges in $\gamma_b$ is less than $ T_{G^K}(\gamma)/K$, we obtain
$$  \Card{\gamma'\setminus \gamma} \leq \rho_d \left(  \frac{N^d T_{G^K}(\gamma)}{K} + N \sum_{ C \in \Bad: \; C\cap \Gamma \ne \varnothing} \Card{  C}\right) \,. $$
\end{proof}

\begin{lemme}
\label{LEMetape2prime} 
Suppose that $G(\{0\})<p_c(d)$. For every $\varepsilon>0$ there exists $p_1 (\epsilon)>0$ and $A (\epsilon)>0$ such that for every $K \ge M_0$, 
for all x large enough,
$$ \P \left( \wT_{G}^{\mathcal C_{G,M_0}} (0,x) \le \wT_{G^K}^{\mathcal C_{G,M_0}} (0,x)\left( 1+\frac{A(\epsilon)}{K}\right) +\varepsilon \|x\|_1 \right)\ge p_1 (\epsilon).$$
\end{lemme}

\begin{proof}
Let $\varepsilon>0$ be fixed. Let $p_c^{\textrm{site}}(d)$ be the critical parameter for independent Bernoulli site percolation on $\mathbb{Z}^d$. Choose $\alpha = \alpha(\epsilon)>0$ and then $\mathfrak p = \mathfrak p(\epsilon) \in (p_c^{\textrm{site}}(d),1)$, such that
\begin{align}
7^d \exp(-\alpha \varepsilon) & \le\frac{1}{3}, \label{choixalpha2} \\
\mathfrak p+\frac{e^\alpha 7^d (1-\mathfrak p)}{1-e^\alpha 7^d (1-\mathfrak p) } &\le \frac{3}{2} \label{choixp2}
\end{align}
as in Lemma \ref{lem:new}. Let $N = N(p_0, q, \mathfrak p (\epsilon))$ be large enough to have the stochastic comparison of Lemma~\ref{couplage} with this parameter $\mathfrak p (\epsilon)$. Let $K \ge M_0$. Fix a large $x$, at least large enough so that $\|x\|_1 \geq 12 d N$.

Let $F_x$ be the following good event: the $N$-boxes containing $0$ and $x$ and all the adjacent boxes are good and belong to an infinite cluster of good boxes. For any $y$ in the same $(3N)$-box as $0$, for any $z$ in the same $(3N)$-box as $x$, let $E_{y,z}$ be the event that $y\in\C_{G,M_0}$, $z\in \C_{G,M_0}$ and the $N$-boxes containing $y$ and $z$ are good and belong to an infinite cluster of good boxes. For any such $(y,z)$, we have $\| y-z\|_1 \leq \|x\|_1 + 6 dN \leq 2 \|x\|_1 $. For any given $\beta''$, we have
\begin{align}
\label{eqnew1}
& \P \left( \wT_{G}^{\mathcal C_{G,M_0}} (0,x) \ge \wT_{G^K}^{\mathcal C_{G,M_0}} (0,x)\left( 1+\frac{\rho_d' N (\epsilon)^d}{K}\right) + 4 \,\varepsilon \,\tilde C_d \, \rho'_d\,\beta'' \|x\|_1 \right) \nonumber \\
&  \leq \PP[F_x^c]  \nonumber\\
&+ \PP \left( F_x \cap \left\{\wT_{G}^{\mathcal C_{G,M_0}} (0,x) \ge \wT_{G^K}^{\mathcal C_{G,M_0}} (0,x)\left( 1+\frac{\rho_d' N (\epsilon)^d}{K}\right) + 4 \,\varepsilon \,\tilde C_d \, \rho'_d\,\beta'' \|x\|_1 \right\} \right) \nonumber\\
&  \leq \PP[F_x^c]  \nonumber\\
&+ \sum_{y,z}\PP \left( \begin{array}{c} F_x \cap \{ \widetilde{0}^{\mathcal C_{G,M_0}} = y \,,\, \widetilde{x}^{\mathcal C_{G,M_0}} = z  \} \\ \cap  \left\{T_{G} (y,z) \ge T_{G^K} (y,z)\left( 1+\frac{\rho_d' N (\epsilon)^d}{K}\right) + 4 \,\varepsilon \,\tilde C_d \, \rho'_d\,\beta'' \|x\|_1 \right\} \end{array} \right) \nonumber\\
& \leq \PP[F_x^c] \nonumber\\
&+ \sum_{y,z} \PP \left(E_{y,z} \cap \left\{T_{G} (y,z) \ge T_{G^K} (y,z)\left( 1+\frac{\rho_d' N (\epsilon)^d}{K}\right) + 2 \,\varepsilon \,\tilde C_d \, \rho'_d\,\beta'' \|y-z \|_1 \right\} \right) \,,
\end{align}
where the sum is over every $y$ in the same $(3N)$-box as $0$, and every $z$ in the same $(3N)$-box as $x$ - indeed, on the event $F_x$, we know that $\C_{G,M_0}$ intersects the box of $0$ (resp. $x$) thus $\widetilde{0}^{\mathcal C_{G,M_0}}$ (resp. $\widetilde{x}^{\mathcal C_{G,M_0}}$) belongs to the same $(3N)$-box as $0$ (resp. $x$). Note that the stochastic comparison and the FKG inequality ensure that
\begin{equation}
\label{eqnew2}
\P(F_x) \geq \theta_{\textrm{site},\mathfrak p(\epsilon) }^{2\dot 3^d} >0\,,
\end{equation}
where $\theta_{\textrm{site},\mathfrak p(\epsilon) }$ denotes the density of the infinite cluster in a supercritical vertex i.i.d. Bernoulli percolation of parameter $\mathfrak p(\epsilon)$.

Consider a couple $(y,z)$ as in \eqref{eqnew1}. On the event $E_{x,y}$, we have
$$T_{G^K} (y,z) \leq M_0 D^{\mathcal C_{G,M_0}}(y,z) < \infty \,.$$
Let $\gamma_{y,z}$ be a geodesic for $T_{G^K} (y,z)$, and let $\Gamma_{y,z}$ be the lattice animal of the $N$-boxes visited by this geodesic. By Lemma \ref{LEM:modifdeux}, we have
\begin{align*}
&  \PP \left(E_{y,z} \cap \left\{T_{G} (y,z) \ge T_{G^K} (y,z)\left( 1+\frac{\rho_d' N (\epsilon)^d}{K}\right) + 2 \,\varepsilon \,\tilde C_d \, \rho'_d\,\beta'' \|y-z \|_1 \right\} \right)  \\
& \leq \PP \left( E_{y,z} \cap \left\{ \sum_{C \in \Bad: \; C \cap \Gamma_{y,z} \ne \varnothing} \Card{C} \geq  \frac{2 \,\varepsilon \,\tilde C_d \, \beta'' \|y-z \|_1}{N(\epsilon)} \right\}   \right)\,.
\end{align*}
Note that by construction, on the event $E_{y,z}$, we have $ |\Gamma_{y,z}| \ge \|y-z\|_1/N$. On the other hand Lemma \ref{lemcard3} implies that $|\Gamma_{y,z}| \leq \tilde C_d (1+  (|\gamma_{y,z}|+1 ) /N ) -1  \leq 2 \tilde C_d |\gamma_{y,z}|/N $ at least for $x$ large enough (remember that $|\gamma_{y,z}| \geq \|y-z\|_1 \geq \|x\|_1- 6dN \geq \|x\|_1 /2$). Thus we obtain
\begin{align}
\label{eqnew3}
&  \PP \left(E_{y,z} \cap \left\{T_{G} (y,z) \ge T_{G^K} (y,z)\left( 1+\frac{\rho_d' N (\epsilon)^d}{K}\right) + 2 \,\varepsilon \,\tilde C_d \, \rho'_d\,\beta'' \|y-z \|_1 \right\} \right) \nonumber \\
& \leq \PP \left( E_{y,z} \cap \{ |\gamma_{y,z}| > \beta'' \|y-z\|_1 \} \right) \nonumber \\
& \qquad + \PP \left( E_{y,z} \cap \left\{ \sum_{C \in \Bad: \; C \cap \Gamma_{y,z} \ne \varnothing} \Card{C} \geq   |\Gamma_{y,z}| \right\}   \right)\,.
\end{align}
Since $G^K(\{0\})=G(\{0\})<p_c(d)$, there exist positive constants $A',B',\beta'$ such that  for all $k\in \NN^*$ (see Proposition 5.8 in Kesten~\cite{Kesten:StFlour}):
\begin{equation}
\label{eqnew4}
\PP \left[ \exists r \textrm{ s.a. path starting at } y \textrm{ s.t. } |r| \geq k \textrm{ and } T_{G^K}(r) \leq \beta' k \right] \leq  A' \exp(-B'k).
\end{equation}
Let $\beta$ be given by Antal and pisztora's estimate \eqref{EQ:AP}. By \eqref{EQ:AP} we have
\begin{align}
\label{eqnew5}
\PP (E_{y,z} \cap \{ D^{\mathcal C_{G,M_0}}(y,z) \geq \beta \| y-z \|_1\} ) &\,\leq  \P( \beta \|y-z\|_1 \leq D^{\mathcal C'_{G,M_0}}(y,z) < + \infty  ) \nonumber \\
&\,\leq\, \hat A\exp(-\hat B \|y-z\|_1) \,.
\end{align}
Fix $\beta''=\frac{ \beta M_0}{\beta'}>0$. Combining \eqref{eqnew4} and \eqref{eqnew5} we obtain the existence of positive constants $A'',B''$ such that
\begin{align}
\label{eqnew6}
 &\PP \left( E_{y,z} \cap \{ |\gamma_{y,z}| >\beta'' \|y-z\|_1 \} \right) \nonumber\\
 & \quad \,\leq \, \PP (E_{y,z} \cap \{ D^{\mathcal C_{G,M_0}}(y,z) \geq \beta \| y-z \|_1\} )\nonumber\\
 & \qquad +  \PP \left( E_{y,z} \cap \{ T_{G^K} (y,z) \leq M_0 \beta \|y-z\|_1  \}\cap \{ |\gamma_{y,z}| >\beta'' \|y-z\|_1 \} \right) \nonumber \\
& \quad\,\leq \,\hat A e^{-\hat B \|y-z\|_1} + A' e^{-B' \beta'' \|y-z\|_1}  \,\leq\, A'' e^{-B'' \|y-z\|_1 }\,.
\end{align}
By Lemma \ref{lem:new}, with the choices \eqref{choixalpha2} and \eqref{choixp2} we made for $\alpha$ and $\mathfrak p$, we know that
\begin{align}
\label{eqnew7}
& \PP \left( E_{y,z} \cap \left\{ \sum_{C \in \Bad: \; C \cap \Gamma_{y,z} \ne \varnothing} \Card{C} \geq   |\Gamma_{y,z}| \right\}   \right)\nonumber\\
& \quad \,\leq \, \P \left( \exists \Gamma \in \Animals, \;  
|\Gamma| \ge \frac{\|y-z\|_1}{N(\epsilon)} , \; \sum_{ C \in \Bad: \; C \cap \Gamma \ne \varnothing}  |C| \ge \varepsilon \Card{\Gamma}
\right)=\mathcal P (\|y-z\|_1) \nonumber\\
& \quad \,\leq \, 2^{-\frac{\|y-z\|_1}{N(\epsilon)}+1}.
\end{align} 
Combining \eqref{eqnew1}, \eqref{eqnew2}, \eqref{eqnew3}, \eqref{eqnew6} and \eqref{eqnew7}, we obtain that
\begin{align*}
\label{eqnew1}
& \P \left( \wT_{G}^{\mathcal C_{G,M_0}} (0,x) \ge \wT_{G^K}^{\mathcal C_{G,M_0}} (0,x)\left( 1+\frac{\rho_d' N (\epsilon)^d}{K}\right) + 4 \,\varepsilon \,\tilde C_d \, \rho'_d\,\beta'' \|x\|_1 \right) \\
& \quad \,\leq \, 1- \theta_{\textrm{site},\mathfrak p(\epsilon) }^{2\dot 3^d} + \sum_{y,z} \left( A'' e^{-B'' \|y-z\|_1 }+ 2^{-\frac{\|y-z\|_1}{N(\epsilon)}+1} \right) \\
& \quad \,\leq \, 1- \theta_{\textrm{site},\mathfrak p(\epsilon) }^{2\dot 3^d} + 2 (3N(\epsilon))^d \left( A'' e^{-B'' \|x\|_1/2 }+ 2^{-\frac{\|x\|_1}{2 N(\epsilon)}+1} \right) \\
& \quad \le 1-p_1 (\epsilon),
\end{align*}
for a well-chosen $p_1 (\epsilon)>0$ and every $x$ large enough. 
\end{proof}

\subsection{Proof of Theorem~\ref{propetape2}}

If $G(\{0\})\ge p_c(d)$, then $\mu_{G^K} (x)=\mu_{G} (x)=0$, so there is nothing to prove. Suppose from now on that $G(\{0\})<p_c(d)$.

For any $\epsilon >0$, consider $p_1(\epsilon)$ and $A(\epsilon)$ as given by Lemma~\ref{LEMetape2prime}, and define, for
$K\ge M_0$, $\Psi(K)=\inf_{\epsilon>0} \frac{A(\epsilon)}{K}+\epsilon$.
It is easy to see that $\lim_{K\to +\infty} \Psi(K)=0$.
Fix $\varepsilon>0$, $\delta>0$, $K \ge M_0$ and $x \in \Zd$.

With the convergence (2) in Proposition~\ref{thmcv} and Lemma~\ref{LEMetape2prime}, we can choose $n$ large enough such that
\begin{align*}
\P \left( \mu_G (x) -\delta \le \frac{\wT_G^{\C_{G,M_0}} (0,nx)}{n}  \right) & \ge 1-\frac{p_1(\epsilon)}{3}, \\
\P \left( \frac{\wT_{G^K}^{\C_{G,M_0}} (0,nx)}{n} \le \mu_{G^K} (x) +\delta \right) & \ge 1-\frac{p_1(\epsilon)}{3}, \\
\P \left( \wT_{G}^{\mathcal C_{G,M_0}} (0,nx) \le \wT_{G^K}^{\mathcal C_{G,M_0}} (0,nx)\left( 1+\frac{A(\epsilon)}{K}\right) +\varepsilon n\|x\|_1 \right) & \ge p_1(\epsilon).
\end{align*}
For every $\epsilon >0$, for every $\delta >0$, on the intersection of these 3 events, that has positive probability, we obtain
\begin{align*}
\forall K  \ge M_0,x\in\Zd  \quad \mu_G (x) -\delta \le (\mu_{G^K} (x)+\delta) \left( 1+\frac{A(\epsilon)}{K}\right) + \varepsilon \|x\|_1,
\end{align*}
and by letting $\delta$ going to $0$ we get
$$\forall \epsilon >0 ,K\ge M_0,x\in\Zd\quad  \mu_G (x)  \le \mu_{G^K} (x) \left( 1+\frac{A(\epsilon)}{K}\right)+\varepsilon \|x\|_1.$$
It follows that for every $\epsilon >0$,
\begin{align*}
0\le \mu_{G^K} (x)-\mu_G(x)&\le \mu_{G^K}\frac{A(\epsilon)}{K}+\epsilon\|x\|_1
\le (\mu_{G}(x)+\|x\|_1)\left( \frac{A(\epsilon)}{K}+\varepsilon \right)\,,
\end{align*}
thus, by optimizing $\epsilon$,
$$0\le \mu_{G^K} (x)-\mu_G(x) \le (\mu_{G}(x)+\|x\|_1)\Psi(K) \,.$$
Theorem~\ref{propetape2} is proved by using the fact that $\displaystyle \lim_{K\to +\infty} \Psi(K)=0$.


\section{Continuity of the time constant, proof of Theorem \ref{thmcont}}
\label{seccontfpp}

We first state two properties:

\begin{lemme}
\label{propetape1}
Suppose that $G, (G_n)_{n\in \NN}$ are probability measures on $[0,+\infty]$ such that $G(\RRp) >p_c(d)$ and $G_n(\RRp) >p_c(d)$ for all $n\in \NN$. If $G_n \overset{d}{\rightarrow} G$ and $G_n  \gestoch G$ for all $n$, then
$$ \forall x\in \ZZ^d\,,\quad \miniop{}{\limsup}{n\rightarrow +\infty } \mu_{G_n} (x)\, \leq\, \mu_G(x) \,.$$
\end{lemme}

\begin{lemme}
\label{propetape3}
Suppose that $G, (G_n)_{n\in \NN}$ are probability measures on $[0,R]$ for some common and finite $R\in \RRp$. If $G_n \overset{d}{\rightarrow} G$, then
$$ \forall x\in \ZZ^d\,,\quad \lim_{n\rightarrow \infty} \mu_{G_n} (x) \,=\, \mu_G (x) \,. $$
\end{lemme}

To prove Theorem~\ref{thmcont}, we follow the general structure of Cox and Kesten's proof of the continuity of the time constant in first-passage percolation with finite passage times in \cite{MR633228}. We first deduce Theorem \ref{thmcont} from Theorem \ref{propetape2} and Lemmas~\ref{propetape1} and~\ref{propetape3}. Lemmas~\ref{propetape1} and~\ref{propetape3} will be respectively proved in subsections \ref{secetape1} and \ref{secetape3}.


\subsection{Proof of Theorem~\ref{thmcont}}

Let $G, (G_n)_{n\in \NN}$ be probability measures on $[0, +\infty]$. We first prove that for all fixed $x\in \ZZ^d$, we have 
\begin{equation}
\label{eqpreuve}
\lim_{n\rightarrow \infty} \mu_{G_n} (x) = \mu_G (x) \,.
\end{equation}
We define $\underline{\fG}_n = \min \{\fG, \fG_n \}$ (resp. $\overline{\fG}_n = \max \{\fG, \fG_n\}$), and we denote by $\underline{G}_n$ (resp. $\overline{G}_n$) the corresponding probability measure on $[0,+\infty]$. Then $\underline{\fG}_n \leq \fG \leq \overline{\fG}_n$ (resp. $\underline{G}_n \leq G_n \leq \overline{G}_n$), thus by Lemma~\ref{propmon} we have $\mu_{\underline{G}_n} (x) \leq \mu_G (x) \leq \mu_{\overline{G}_n} (x)$. 
To conclude that (\ref{eqpreuve}) holds, it is sufficient to prove that
$$ (i) \quad \miniop{}{\liminf}{n\rightarrow \infty} \mu_{\underline{G}_n} (x) \geq \mu_G (x)\quad \textrm{and} \quad (ii) \quad\miniop{}{\limsup}{n\rightarrow +\infty} \mu_{\overline{G}_n}(x) \leq \mu_G(x) \,.$$
Notice that $\overline{G}_n \overset{d}{\rightarrow} G$ and $\underline{G}_n \overset{d}{\rightarrow} G$. Inequality $(ii)$ is obtained by a straightforward application of Lemma~\ref{propetape1}. For any $K \in \RRp$, we define $G^K = \mathbbm{1}_{[0,K)} G + G([K,+\infty]) \delta_K$ (resp. $\underline{G}_n^K= \mathbbm{1}_{[0,K)} \underline{G}_n + \underline{G}_n([K,+\infty]) \delta_K$), the distribution of $t_G^K(e) = \min (t_{G} (e) , K)$ (resp. $t_{\underline{G}_n}^K(e) = \min (t_{\underline{G}_n} (e) , K)$). Using Lemmas~\ref{propmon} and \ref{propetape3}, since \smash{$\underline{G}_n^K \overset{d}{\rightarrow} G^K$}, we obtain for all $K$
$$\miniop{}{\liminf}{n\rightarrow \infty} \mu_{\underline{G}_n}(x) \,\geq \, \lim_{n\rightarrow \infty} \mu_{\underline{G}_n^K}(x) \,=\, \mu_{G^K} (x) \,,$$
and by Theorem \ref{propetape2} we have $\lim_{K\rightarrow \infty}\mu_{G^K}(x) = \mu_G(x) $. This concludes the proof of $(i)$, and of \eqref{eqpreuve}.

By homogeneity, (\ref{eqpreuve}) also holds for all $x\in \QQ^d$. We know that $|\mu_{G_n} (x) - \mu_{G_n} (y)| \leq \mu_{G_n} (e_1) \| x-y \|_1$, where $e_1 = (1,0,\dots, 0)$. Moreover $\lim_{n\rightarrow \infty} \mu_{G_n} (e_1) = \mu_G (e_1)$, thus for all $n\geq n_0$ large enough we have $|\mu_{G_n} (x) - \mu_{G_n} (y)| \leq 2 \mu_{G} (e_1) \| x-y \|_1$ for all $x,y\in \R^d$. This implies that for any fixed $\eps >0$, there exists $\eta>0$ such that for all $x,y\in \R^d$ such that $\|x-y\|_1\leq \eta$, we have
$$ \sup \{ |\mu_{G} (x) - \mu_{G} (y)| , |\mu_{G_n} (x) - \mu_{G_n} (y)| , n\geq n_0 \} \,\leq \, \eps \,.$$
There exists a finite set $(y_1,\dots , y_m)$ of rational points of $\R^d$ such that 
$$ \SS^{d-1} \,\subset \, \bigcup_{i=1}^m \{ x\in \R^d \,:\, \|y_i - x\|_1 \leq \eta \} \,.$$
Thus
$\displaystyle  \miniop{}{\limsup}{n\rightarrow +\infty} \sup_{x\in \SS^{d-1}} | \mu_{G_n} (x) - \mu_G (x)| \leq  2 \eps + \lim_{n\rightarrow +\infty} \max_{i=1,\dots ,m}  | \mu_{G_n} (y_i) - \mu_G (y_i)| = 2 \eps .$
Since $\eps$ was arbitrary, Theorem \ref{thmcont} is proved.


\subsection{Bound on sequences of probability measures}
\label{secdistfunc}

\begin{lemme}
\label{lemrepart}
Suppose that $G$ and $(G_n)_{n\in \NN}$ are probability measures on $[0,+\infty]$ such that $G_n \overset{d}{\rightarrow} G$.
\begin{itemize}
\item[$(i)$] If $G(\RRp) >p_c(d)$ and $G_n(\RRp) >p_c(d)$ for all $n\in \NN$, then there exists a probability measure $H^+$ on $[0,+\infty]$ such that $G_n\lestoch H^+$ for all $n$ and $H^+(\RRp) > p_c(d)$.
\item[$(ii)$] If $G(\{0\}) < p_c(d)$ and $ G_n (\{0\}) <p_c(d)$ for all $n\in \NN$, then there exists a probability measure $H^-$ on $[0,+\infty]$ such that $G_n\gestoch H^-$ for all $n$ and $H^-(\{0\})< p_c(d)$.
\end{itemize}
\end{lemme}

\begin{proof} $(i)$
We define $\hat \fH^+ = \sup_{n\in \NN} \fG_n$, and $\fH^+(x) = \inf\{ \hat \fH^+(y) \,:\, y<x  \}$ for all $x\in \RRp$. Then $\hat \fH^+$ and $\fH^+$ are non-increasing functions defined on $\RRp$ and they take values in $[0,1]$. By construction $\fH^+$ is left continuous and $\fH^+\geq \fG_n$, for all $n\in \NN$. Moreover we have $\hat \fH^+ (x) = \fH^+ (x) = 1 $ for all $x\leq 0$. Thus there exists a probability measure $H^+$ on $[0,+\infty]$ such that $\fH^+(t) = H^+([t, +\infty])$ for all $t\in \RRp$. It remains to prove that $H^+(\RRp) > p_c(d)$. Since $G(\RRp) >p_c(d)$, i.e. $\lim_{+\infty} \fG < 1-p_c(d)$, there exist $A\in \RRp$ and $\eps>0$ such that $\fG$ is continuous at $A$ and $\fG(A) \leq 1- p_c(d) - 2\eps$. Moreover \smash{$G_n \overset{d}{\rightarrow} G$} and $\fG$ is continuous at $A$, thus there exists $n_0$ such that for all $n\geq n_0$ we have $\fG_n(A) \leq \fG(A) + \eps \leq 1 - p_c(d) - \eps$. For any $i\in \{1,\dots , n_0-1\}$, $G_i(\RRp) >p_c(d)$ thus there exists $A_i < +\infty$ such that $\fG_i (A_i) < 1- p_c(d)$. Fix $A' = \max (A, A_0, \dots,A_{n_0-1}) <+\infty $. We conclude that
\begin{align*}
 \hat \fH^+(A') & \,=\, \max \left( \fG_0 (A') , \dots , \fG_{n_0-1} (A'), \sup_{n\geq n_0} \fG_n (A') \right)\\
 & \,\leq\,  \max \left( \fG_0 (A_0) , \dots , \fG_{n_0-1} (A_{n_0-1}), \sup_{n\geq n_0} \fG_n (A) \right)  \,<\, 1- p_c(d) \,,
\end{align*} 
thus $H^+(\RRp) =  1- \lim_{+\infty} \fH^+ > p_c(d)$.

\noindent
$(ii)$ We define $\fH^- = \inf_{n\in \NN} \fG_n$. Then $\fH^-$ is non-increasing, defined on $\RRp$ and it takes values in $[0,1]$. Fix $t_0 \in \RRp$. Let us prove that $\fH^-$ is left continuous at $t_0$. By definition of $\fH^-$, for any $\eps>0$, there exists $n_0$ such that $\fH^-(t_0) \geq \fG_{n_0} (t_0) - \eps$. Since $\fG_{n_0}$ is left continuous, there exists $\eta>0$ such that for all $t\in (t_0 -\eta, t_0 ]$ we have $\fG_{n_0} (t) \leq \fG_{n_0} (t_0) + \eps$. Thus for all $t\in (t_0-\eta, t_0 ]$, we obtain
$$ \fH^-(t) \,\leq \, \fG_{n_0} (t) \,\leq\, \fG_{n_0} (t_0) + \eps \,\leq\, \fH^-(t_0) + 2 \eps \,, $$
thus $\fH^-$ is right continuous. By construction $\fH^- \leq \fG_n$, for all $n\in \NN$. Moreover $\fH^-(t) =1$ for all $t\leq 0$. Thus there exists a probability measure $H^-$ on $[0,+\infty]$ such that $\fH^-(t) = H^-([t, +\infty])$ for all $t\in \RRp$. It remains to prove that $H^-(\{0\})<p_c(d)$. Since $G(\{0\})<p_c(d)$, there exists $\eta >0$ such that $G([0,\eta))<p_c(d)$, i.e., $\fG(\eta)>1-p_c(d)$. Let $\eps >0$ such that $\fG(\eta)\geq 1-p_c(d) + 2 \eps$. There exists $\delta \in [0,\eta)$ such that $\fG$ is continuous at $\delta$. Then $\lim_{n\rightarrow \infty }\fG_n (\delta) = \fG (\delta)$, thus there exists $n_0$ such that for all $n\geq n_0$, $\fG_n(\delta) \geq \fG (\delta) - \eps \geq 1- p_c(d) + \eps$. For any $i\in \{1,\dots , n_0-1\}$, there exists $\delta_i>0$ such that $\fG_i (\delta_i) > 1- p_c(d)$. Fix $\delta' = \min (\delta, \delta_0, \dots, \delta_{n_0-1}) >0$. We conclude that
\begin{align*}
\fH^- (\delta') & \,=\, \min \left(\fG_0 (\delta'), \dots , \fG_{n_0 -1} (\delta') , \inf_{n\geq n_0} \fG_n(\delta') \right)\\
 & \,\geq \, \min \left(\fG_0 (\delta_0), \dots , \fG_{n_0 -1} (\delta_{n_0-1}) , \inf_{n\geq n_0} \fG_n(\delta) \right)\,>\, 1- p_c(d)\,, 
 \end{align*}
and 
$H^-(\{0\}) \,=\, 1- \lim_{t\rightarrow 0, t>0} \fH^-(t) \, \leq \,1- \fH (\delta') < p_c(d) \,.$
\end{proof}


\subsection{Proof of Lemma~\ref{propetape1}}
\label{secetape1}

We follow the structure of Cox and Kesten's proof of Lemma 1 in \cite{MR633228}.

We take $H^+$ as given in Lemma \ref{lemrepart} (i), and we fix $M\in \RRp$ such that $H^+([0,M]) > p_c(d)$. We work with the stabilized points $\wx^{\C_{H^+,M}}$. We consider a point $x \in \ZZ^d$, and $k\in \NN^*$. For any path $r$ from $\w0^{\C_{H^+,M}}$ to \smash{$\wkx^{\C_{H^+,M}}$}, using Lemma~\ref{propCV} we have a.s.
$$ T_G (r)\,=\, \sum_{e\in r} t_G(e) \,=\, \lim_{n\rightarrow +\infty}  \sum_{e\in r} t_{G_n}(e) \,\geq \, \miniop{}{\limsup}{n\rightarrow +\infty} \wT_{G_n}^{\C_{H^+,M}} (0, kx) \,.$$
Taking the infimum over any such path $r$, we obtain 
$$\wT_{G}^{\C_{H^+,M}} (0, kx) \,\geq \,  \miniop{}{\limsup}{n\rightarrow +\infty} \wT_{G_n}^{\C_{H^+,M}} (0, kx) \,. $$
Conversely, since $G \lestoch G_n$, thanks to the coupling of the laws we get $\wT_{G}^{\C_{H^+,M}} (0, kx) \leq \wT_{G_n}^{\C_{H^+,M}} (0, kx) $ for all $n$, thus
$$ \forall k \in \NN^*\,,\textrm{ a.s.,} \quad  \lim_{n\rightarrow \infty} \wT_{G_n}^{\C_{H^+,M}} (0, kx) \,=\,  \wT_{G}^{\C_{H^+,M}} (0, kx) \,.$$
Since for all $n$ we have $\wT_{G_n}^{\C_{H^+,M}} (0, kx) \leq \wT_{H^+}^{\C_{H^+,M}} (0, kx)$ that is integrable by Proposition~\ref{propmoments}, the dominated convergence theorem implies that, for all $k\in \NN^*$,
\begin{equation}
\label{eqetape11}
 \lim_{n\rightarrow \infty} \EE\left[ \wT_{G_n}^{\C_{H^+,M}} (0, kx) \right] \,=\, \EE \left[  \wT_{G}^{\C_{H^+,M}} (0, kx)  \right]\,.
 \end{equation}
By Lemma~\ref{propstab}, we know that $\mu_G(x)= \inf_{k\in \NN^*}  \EE \left[ \wT_{G}^{\C_{H^+,M}} (0, kx) \right]/k$. For any $\eps>0$, there exists $K (G, \eps)$ such that
\begin{equation}
\label{eqetape12}
\mu_G(x) \,\geq \, \frac{ \EE \left[ \wT_{G}^{\C_{H^+,M}} (0, Kx) \right]}{K} -\eps \,,
\end{equation}
and using (\ref{eqetape11}) we know that there exists $n_0(\eps, K)$ such that for all $n\geq n_0$ we have
\begin{equation}
\label{eqetape13}
 \frac{ \EE\left[ \wT_{G}^{\C_{H^+,M}} (0, Kx) \right]}{K}  \,\geq \,  \frac{ \EE\left[ \wT_{G_n}^{\C_{H^+,M}} (0, Kx) \right]}{K}  - \eps \,.
\end{equation}
Since $\mu_{G_n} (x) =\inf_{k\in \NN^*}  \EE \left[ \wT_{G_n}^{\C_{H^+,M}} (0, kx) \right]/k$, combining equations (\ref{eqetape12}) and (\ref{eqetape13}), we obtain that for any $\eps >0$, for all $n$ large enough,
$$ \mu_G(x) \,\geq\, \mu_{G_n} (x) - 2\eps \,.$$
This concludes the proof of Lemma~\ref{propetape1}.

\begin{rem}
The domination we use to prove (\ref{eqetape11}) is free, since whatever the probability measure $H^+$ on $[0,+\infty]$ we consider, the regularized times $\wT_{H^+}^{\C_{H^+,M}}(0,x)$ are always integrable. In \cite{MR588407}, Cox considered the (non regularized) times $T_{G_n} (0,x)$ for probability measures $G_n$ on $\RRp$. By Lemma \ref{lemrepart} it is easy to obtain $T_{G_n} (0,x) \leq T_H (0,x)$ for some probability measure $H$ on $\RRp$. However, without further assumption, $T_H(0,x)$ may not be integrable. This is the reason why Cox supposed that the family $(G_n, n\in \NN)$ was uniformly integrable. In \cite{MR624685}, Cox and Kesten circumvent this problem by considering some regularized passage times that are always integrable. There is no straigthtforward generalization of their regularized passage times to the case of possibly infinite passage times, but the  $\wT$ introduced in \cite{CerfTheret14} plays  the same role.
\end{rem}


\subsection{Proof of Lemma~\ref{propetape3}}
\label{secetape3}

Of course, Lemma~\ref{propetape3} can be seen as a particular case of the continuity result by Cox and Kesten. But, as noted by Kesten in his Saint-Flour course~\cite{Kesten:StFlour}, the Cox--Kesten way makes use of former results by Cox in \cite{MR624685} and is not the shortest path to a proof in the compact case.
In~\cite{Kesten:StFlour} Kesten also gave a sketch of a shorter proof in the compact case.
We thought the reader would be pleased to have a self-contained proof, so we present a short but full proof of Lemma~\ref{propetape3}, quite inspired by Kesten~\cite{Kesten:StFlour}.

Let $G,(G_n)_{n\in \NN}$ be probability measures on $[0,R]$, and consider $x\in \ZZ^d$. As in the proof of Theorem \ref{thmcont}, we have $G_n \lestoch \overline{G}_n$, where $\overline{\fG}_n= \max (\fG, \fG_n)$, thus $\mu_{G_n} \leq \mu_{\overline{G}_n}$. Applying Lemma~\ref{propetape1}, we know that 
$$\miniop{}{\limsup}{n\rightarrow +\infty} \mu_{\overline{G}_n} (x) \,\leq\, \mu_G(x) \,.$$
If $\mu_G(x) = 0$, then $\lim_{n\rightarrow \infty} \mu_{G_n} (x) =\lim_{n\rightarrow \infty} \mu_{\overline{G}_n} (x) = \mu_G (x) =0$ and the proof is complete. We suppose from now on that $\mu_G(x) >0$, thus $x\neq 0$. Since the passage times $t_G(e)$ are finite, it is well known that $\mu_G(x) >0$ for $x\neq 0$ if and only if $G(\{0\})<p_c(d)$ (see Theorem 6.1 in \cite{Kesten:StFlour}, or Proposition \ref{thmpropmu} in a more general setting). We want to prove that $\liminf_{n\rightarrow \infty} \mu_{\underline{G}_n} (x) \geq \mu_G(x)$, where $\underline{\fG}_n= \min (\fG, \fG_n)$. Notice that $\wx^{\C_{\underline{G}_n,M}}= \wx^{\C_{G,M}} = x$ for any $M\geq R$, thus we do not need to introduce regularized times $\wT$. In what follows we note {\em s.a.} for {\em self avoiding}. Since $\underline{G}_n \overset{d}{\rightarrow} G$, we have $\lim_{n\rightarrow \infty}\underline{G}_n(\{0\}) \leq G(\{0\}) < p_c(d) $, thus we consider only $n$ large enough so that $\underline{G}_n(\{0\})<p_c(d)$. Applying Lemma \ref{lemrepart} (ii) to the sequence of functions $\underline{G}_n$, we obtain the existence of a probability measure $H^-$ on $[0,+\infty]$ (in fact on $[0,R]$) such that $H^-\lestoch \underline{G}_n$ for all $n$ and $H^-(\{0\})<p_c(d)$. Thanks to the coupling, we know that $t_{H^-}(e) \leq t_{\underline{G}_n} (e) \leq t_G(e)$ for every edge $e$, thus we obtain that for all $A\in \NN^*$, for all $C\in \RRp$,
\begin{align*}
&\PP [ T_{\underline{G}_n} (0,kx) \leq  T_G(0,kx) - \eps k ]  \\ 
& \,\leq \,\PP \left[ \exists r \textrm{ s.a. path starting at } 0 \textrm{ s.t. } |r| \geq A k \textrm{ and } T_{\underline{G}_n}(r) \leq AC k \right] \\
& \quad +  \PP [T_{\underline{G}_n} (0,kx) > A C k] + \sum_{\tiny{\begin{array}{c} r \textrm{ s.a. path from }0 \\ \textrm{ s.t. }|r| \leq A k   \end{array}}} \PP \left[ \sum_{e\in r} t_G(e) - t_{\underline{G}_n} (e) \geq \eps k \right] \\
& \,\leq \,\PP \left[ \exists r \textrm{ s.a. path starting at } 0 \textrm{ s.t. } |r| \geq A k \textrm{ and } T_{H^-}(r) \leq A C k \right] \\
& \quad +  \PP [T_{G} (0,kx) > A C k] + (2d)^{A k } \,\PP \left[ \sum_{i=1}^{Ak } t_G(e_i) - t_{\underline{G}_n} (e_i) \geq \eps k \right] \,,
\end{align*}
where $(e_i, i=1,\dots , Ak)$ is a collection of distinct edges. Since $H^-(\{0\})<p_c(d)$, we know that we can choose $C \in (0,+\infty)$ (depending on $d$ and $H$) such that there exist finite and positive constants $D,E$ (depending also on $d$ and $H$) satisfying, for all $k\in \NN^*$,
$$ \PP \left[ \exists r \textrm{ s.a. path starting at } 0 \textrm{ s.t. } |r| \geq k \textrm{ and } T_{H^-}(r) \leq C k \right] \,\leq \, D e^{-Ek} $$
(see Proposition 5.8 in \cite{Kesten:StFlour}). Since the support of $G$ is included in $[0,R]$ for some finite $R$, we know that $T_G (0, kx) \leq R k \|x\|_1$, thus we choose $A$ large enough (depending on $F$, $d$ and $C$) so that 
$$\PP [T_{G} (0,kx) > A C k] \,=\,0\,.$$
If we prove that there exists $n_0 (G, (\underline{G}_n), \eps)$ such that for all $n\geq n_0$,
\begin{equation}
\label{eqetape3}
\sum_{k>0} (2d)^{A k }\,  \PP \left[ \sum_{i=1}^{ Ak } t_G(e_i) - t_{\underline{G}_n} (e_i) \geq \eps k \right] \,<\, +\infty \,,
\end{equation}
then for all $n\geq n_0$ we have $\sum_k \PP [ T_{\underline{G}_n} (0,kx) \leq  T_G(0,kx) - \eps k ]<+\infty$. By Borel-Cantelli's lemma we obtain that for all $n\geq n_0$, a.s., for all $k\geq k_0(n)$ large enough, 
$$T_{\underline{G}_n} (0,kx) > T_G(0,kx) - \eps k \,,$$
thus for all $n\geq n_0$ we get
$$ \mu_{\underline{G}_n} (x) \geq \mu_G (x) - \eps\,. $$
We conclude that $\liminf_{n\rightarrow \infty} \mu_{\underline{G}_n} (x) \geq \mu_G(x)$. It remains to prove (\ref{eqetape3}). For any $\alpha >0$, by Markov's inequality we have
\begin{align*}
(2d)^{A k } \,\PP & \left[ \sum_{i=1}^{ Ak } t_G(e_i) - t_{\underline{G}_n} (e_i) \geq \eps k \right] \\
&\qquad  \qquad \,\leq\, \left( 2d \exp \left( \frac{-\alpha \eps }{A}\right) \EE \left[ \exp \left(\alpha (t_G(e) - t_{\underline{G}_n} (e)) \right) \right]  \right)^{Ak}  \,. 
\end{align*}
By Lemma~\ref{propCV} we have $\lim_{n\rightarrow \infty} t_{\underline{G}_n} (e) = t_G(e)$ a.s. Since $t_{\underline{G}_n} (e), t_G (e) \leq R$ we obtain by a dominated convergence theorem that
$$ \lim_{n\rightarrow \infty} \EE \left[ \exp \left(\alpha (t_G(e) - t_{\underline{G}_n} (e)) \right)\right]\,=\, 1 \,. $$
We choose $\alpha (\eps)$ large enough so that 
$$2d \,\leq \, \exp \left( \frac{\alpha \eps}{4A} \right)\,,$$
and then $n_0(G, (\underline{G}_n), \eps)$ large enough so that for all $n\geq n_0$, we have
$$  \EE \left[ \exp \left(\alpha (t_G(e) - t_{\underline{G}_n} (e)) \right)\right] \,\leq\, \exp \left( \frac{\alpha \eps}{4A} \right)\,.$$
Thus for all $n\geq n_0$, we have 
$$(2d)^{A k } \PP \left[ \sum_{i=1}^{ Ak } t_G(e_i) - t_{\underline{G}_n} (e_i) \geq \eps k \right]  \,\leq\, \exp \left( - \frac{\alpha \eps}{2A} \right)\,, $$
so (\ref{eqetape3}) is proved.


\section{Continuity of the Cheeger constant, proof of Theorem  \ref{thm:isocnt}}
\label{seccheeger}

The definition of the objects used in this section are given in section \ref{secdefcheeger}. The main step in the proof of Theorem \ref{thm:isocnt} is the following lemma:
\begin{lemme}\label{lem:normcont}
For every $p > p_c(2)$,
$\displaystyle \lim_{p' \rightarrow p} \sup_{x\in\mathbb{S}^1}|\beta_{p'}(x)-\beta_p(x)|  \,=\, 0 \,. $
\end{lemme}
\begin{proof}
Let $x\in \SS^1$. Let $p_c(2)<p_0\leq p \leq q $, and define $\delta = q-p$. We couple the percolations with different parameters in the usual way using uniform variables. We extend the definition of $\tilde{y}^{\C}$ to any $y\in \RR^d$. For a path $r\in \mathcal{R} (x,y)$, let us define $\mathbf{b}_p (r) = \Card{\{ e \in \partial^+ r \,:\, e \text{ is } p-open\}}$. For $x,y\in \C_p$, we define $b_p (x,y) = \inf \{ \mathbf{b}_p (r) \,:\, r \in \mathcal{R} (x,y), r\textrm{ is }p-\text{open} \}$.

{\em Step (i).}
By Lemma \ref{lem:rightint} there exist $C,C',\alpha>0$ (depending on $p_0$) such that $\forall p\geq p_0$, $\forall n$,
\begin{equation}\label{eq:shortpaths}
\pr \left[\exists \gamma\in \bigcup_{x\in\mathbb{Z}^2} \mathcal{R}(0,x)\,:\,|\gamma|>n \,,\,\, \mathbf{b}_p (\gamma) \leq \alpha n \right]\le  Ce^{-C'n}
.\end{equation}
Let $F_{p_0}$ be the event $\{ 0 \in \C_{p_0}\} \cap \{nx \in \C_{p_0} \}$. On the event $F_{p_0}$, by \cite[Lemma 2.5]{biskup2012isoperimetry}, we have $b_p (0, nx) \leq 3 D^{\C_p} (0, nx) \leq 3 D^{\C_{p_0}} (0,nx)$, thus using \eqref{EQ:AP3} we know that there exist positive constants $\hat A,\hat B, \beta$ (depending only on $p_0$) such that for all $p>p_0$, for all $x \in \SS^1$,
\begin{equation}
\label{APencore}
\pr \left[ F_{p_0} \cap \{ b_p (0,nx) \geq 3 \beta n  \}  \right] \le \hat A \exp(- \hat B n)\,.
\end{equation}
For any $p$-open path $\gamma$, $\gamma$ is also $q$-open. However some additional right-boundary edges may be open. To bound the difference between $\mathbf{b}_{q}(\gamma)$ and $\mathbf{b}_{p}(\gamma)$, note that if $|\gamma|<\alpha' n$ by \cite[Lemma 2.5]{biskup2012isoperimetry} $|\partial^+\gamma|<3\alpha' n$. We can bound $\mathbf{b}_{q}(\gamma)-\mathbf{b}_{p}(\gamma)$ by Cram\'er's theorem. For every fixed path $\gamma$ such that $|\gamma|<\alpha' n$, for every $\epsilon>0$ and $\delta < \epsilon$,
\begin{equation}
\label{hop}
\pr \left[\mathbf{b}_{q}(\gamma)-\mathbf{b}_{p}(\gamma)>3\epsilon\alpha' n\right]\le e^{-3\alpha' n\left(\epsilon\log{\frac{\epsilon}{\delta}}+(1-\epsilon)\log{\frac{1-\epsilon}{1-\delta}}\right)}.
\end{equation}
Fix $\alpha' = 3 \beta / \alpha$. Since there are at most $4^{\alpha' n}$ paths of length smaller than $\alpha' n$ containing $0$, we obtain that for all $p_0\leq p<q$, for all $x\in \SS^1$,
\begin{align}
& \pr \left[ b_q (\tilde{0}^{\mathcal{C}_q},\widetilde{nx}^{\mathcal{C}_q}) > b_p (\tilde{0}^{\mathcal{C}_p},\widetilde{nx}^{\mathcal{C}_p}) + 3 \epsilon \alpha' n \right]\nonumber \\
& \quad \,\leq\,  \pr [ F_{p_0}^c]  + Q_p [F_{p_0} \cap \{b_p (0,nx) > 3\beta n\}  ]\nonumber \\
& \quad \quad +  \pr \left[  F_{p_0} \cap \{ \exists \gamma \in \mathcal R(0,nx) \,:\, |\gamma| > \alpha' n \,,\,\, \mathbf{b}_{p} (\gamma) \leq 3\beta  n \} \right] \nonumber \\
& \quad \quad +  \pr \left[  F_{p_0} \cap \{ \exists \gamma \in \mathcal R(0,nx) \,:\, |\gamma| \leq \alpha' n \,,\,\, \mathbf{b}_{q} (\gamma) > \mathbf{b}_p (\gamma) + 3 \epsilon \alpha' n \} \right] \nonumber\\
& \quad \,\leq \, (1 - \theta_{p_0}^2) + \hat A e^{- \hat B n} + C e^{-C' \alpha' n} \nonumber\\
& \quad \quad + 4^{\alpha' n} e^{-3\alpha' n\left(\epsilon\log{\frac{\epsilon}{\delta}}+(1-\epsilon)\log{\frac{1-\epsilon}{1-\delta}}\right)}\,. \nonumber
\end{align}
For every $p_0 > p_c(d)$, for every $\epsilon>0$, there exists $\delta (\epsilon)>0$ and $p_2 (\epsilon)>0$ such that for every $x\in\mathbb{S}^1$, for every $p_0 \leq p < q $ satisfying $q-p<\delta$, we have
$$\pr \left[ b_q (\tilde{0}^{\mathcal{C}_q},\widetilde{nx}^{\mathcal{C}_q}) > b_p (\tilde{0}^{\mathcal{C}_p},\widetilde{nx}^{\mathcal{C}_p}) + 3 \epsilon \alpha' n \right] \,\leq\, 1-p_2 (\epsilon)\,,$$
thus for every $p_0 > p_c(d)$, for every $\epsilon>0$, there exists $\delta (\epsilon)>0$ such that for every $x\in\mathbb{S}^1$, for every $p_0 \leq p < q $ satisfying $q-p<\delta$, we have
 \[\beta_{q}(x)<\beta_p(x)+3\alpha' \epsilon.\]

{\em Step (ii).}
Given a $q$-open path $\gamma$, $\gamma$ may not be $p$-open. Thus we use the results of Section \ref{secrenorm} to modify the path to a $p$-open path which does not gain too many extra right-boundary edges. We mimic the proof of Lemma \ref{LEMetape2prime}. Fix $\epsilon>0$. Choose $\alpha (\epsilon)$ and $\mathfrak p (\epsilon)$ as in \eqref{choixalpha} and \eqref{choixp}.
\begin{itemize}
\item If $p$ is fixed and we let $q$ goes to $p$, we choose $p_0 = p$, $q_1\in (p_0, p_0 + \delta_1 (p))$ as defined in Lemma \ref{couplage}, and we consider only values of $q$ such that $q\in [p_0, q_1]$. Then we choose $N = N'(p_0, q_1, \mathfrak p (\epsilon))$ as given in Lemma \ref{couplage}.\\
\item If $q$ is fixed and we let $p$ goes to $q$, we choose $p_0 \in ( p_c(d), q]$ such that $q-p_0 \leq \delta_0 (q)$ as defined in Lemma \ref{couplage}, and we consider only values of $p$ in the interval $[p_0, q]$. Then we choose $N = N(p_0, q, \mathfrak p (\epsilon))$ as given in Lemma \ref{couplage}.
\end{itemize}
Let $x\in \mathbb{S}^1$, we denote by $\lfloor nx \rfloor$ the point $y$ of $\ZZ^d$ which  minimizes $\|nx-y\|_1$ (with a deterministic rule to break ties). Let $F'$ be the following good event: the $N$-boxes containing $0$ and $\lfloor nx \rfloor$ and all the adjacent boxes are good and belong to an infinite cluster of good boxes. By the FKG inequality and the stochastic comparison, we have
\begin{equation}
\label{eqnewbis1}
 \PP (F') \,\geq \theta_{\textrm{site},\mathfrak p(\epsilon) }^{18} \,.
 \end{equation}
Fix $\alpha '' = 6 \beta / \alpha = 2 \alpha'$ as defined in step $(i)$. We have
\begin{align}
\label{eqnewbis2}
& \pr \left[ b_p (\tilde{0}^{\mathcal{C}_{p_0}},\widetilde{\lfloor nx\rfloor}^{\mathcal{C}_{p_0}}) > b_q (\tilde{0}^{\mathcal{C}_{p_0}},\widetilde{\lfloor nx\rfloor}^{\mathcal{C}_{p_0}}) + 12 \alpha'' \tilde{C}_d \rho_d \epsilon n \right]\nonumber \\
& \,\leq \, \PP[F^c] + \sum_{y,z} \PP\left( E_{y,z} \cap \{  b_p (y,z) > b_q (y,z) + 12 \alpha'' \tilde{C}_d \rho_d \epsilon n\} \right)\,,
\end{align}
where the sum is over every $y$ in the same $(3N)$-box as $0$ and $z$ in the same $(3N)$-box as $\lfloor nx\rfloor$, and $E_{y,z}$ is the event that $y\in\C_{p_0}, z\in \C_{p_0}$ and the $N$-boxes containing $y$ and $z$ are good and belong to an infinite cluster of good boxes. For any such $(y,z)$, on $E_{y,z}$, let $\gamma_{y,z} \in \mathcal R (y,z)$ be a $q$-open right-most path from $y$ to $z$ such that $b_q (y,z) = \mathbf{b}_q (\gamma_{y,z})$, and let $\Gamma_{y,z}$ be the lattice animal of $N$-boxes it visits. For short, we write $\gamma$ for $\gamma_{y,z}$. As previously we define
\begin{align*}
\gamma_a & =\{e\in\gamma: e\text{ is }p\textrm{-open}\}\\
\gamma_b & =\{e\in\gamma: e\text{ is }p\textrm{-closed}\}\,.
\end{align*}
By Lemma \ref{LEM:modif}, on the event $E_{y,z}$, there exists a path $\gamma'$  with the following properties: 
\begin{enumerate}
\item $\gamma'$ is a path from $y$ to $z$ which is $p$-open;
\item $\gamma'\setminus \gamma$ is a collection of simple paths (and also right-most) that intersect $\gamma' \cap \gamma$ only at their endpoints thus $\gamma'$ is a right-most path (see \cite[Lemma 2.6]{biskup2012isoperimetry});
\item $\Card{\gamma' \setminus \gamma} \,\leq \, \rho_d \left(  N^d \Card{\gamma_b} + N \sum_{ C \in \Bad: \; C\cap \Gamma \ne \varnothing} \Card{  C}
\right) \,.$
\end{enumerate}
Note that 
\begin{equation*}
\mathbf{b}_{p}(\gamma')\le \mathbf{b}_{q}(\gamma')\le \mathbf{b}_{q}(\gamma)+3|\gamma'\setminus\gamma|
.\end{equation*}
Moreover, since a simple path is also a right-most path, we have for all $y,z\in \C_{p_0}$,
\begin{equation}
\label{eqnewbis3}
b_q (y,z) \,\leq \, 3 D^{\C_{p_0}} (y,z) \,.
\end{equation}
Using Equation \eqref{eqnewbis3}, Proposition \ref{propbiskup}, Antal and Pizstora's estimate \eqref{EQ:AP}, Cram{\'e}r's theorem again and Lemma \ref{lem:new}, for all $x\in \SS^1$ and for all $n$ large enough (in particular such that $\|y-z\|_1 \leq \|\lfloor nx \rfloor \| + 12 N \leq 2n $ and $\|y-z\|_1 \geq \|\lfloor nx \rfloor \| - 12 N \geq n/2 $), we have
\begin{align}
\label{eqnewbis4}
&\PP\left( E_{y,z} \cap \{  b_p (y,z) > b_q (y,z) + 12 \alpha'' \tilde{C}_d \rho_d \epsilon n\} \right)\nonumber \\
& \quad \,\leq\, \pr [E_{y,z} \cap \{ b_q (y,z) > 6 \beta n\}  ]\nonumber \\
& \quad \quad +  \pr \left[  E_{y,z}\cap \{ \exists \gamma \in \mathcal R(y,z) \,:\, |\gamma| > \alpha'' n \,,\,\, \mathbf{b}_{p_0} (\gamma)\leq \mathbf{b}_{q} (\gamma) \leq 6\beta  n \} \right] \nonumber \\
& \quad \quad +  \pr \left[  E_{y,z} \cap \{ \exists \gamma \in \mathcal R(y,z) \,:\, \gamma \textrm{ is }q\textrm{-open}, |\gamma| \leq \alpha'' n \,,\,\, |\gamma'\setminus\gamma| > 4 \alpha'' \tilde{C}_d \rho_d \epsilon n \}  \right] \nonumber\\
& \quad \,\leq \,  \PP \left( \beta \|y-z\|_1 \leq 2 \beta n \leq D^{\C_{p_0}} (y,z) < \infty \right) + C e^{-C' \alpha'' n} \nonumber\\
& \quad \quad +  \pr \left[  \exists \gamma  \,:\, \gamma \textrm{ starts at $y$, $\gamma$ is $q$-open}, |\gamma| \leq \alpha'' n \,,\,\, |\gamma_b| > \frac{2 \alpha'' \tilde{C}_d \epsilon n }{N^d} \right]  \nonumber\\
& \quad \quad + \pr \left[   \exists \Gamma \in \Animals, \;  
|\Gamma| \ge \frac{n}{N} , \; \sum_{ C \in \Bad: \; C \cap \Gamma \ne \varnothing}  |C| \ge \varepsilon \Card{\Gamma}
\right] \nonumber\\
& \quad \,\leq \,  \hat A e^{- \hat B \|y-z\|_1} + C e^{-C' \alpha'' n} \nonumber\\
& \quad \quad + 4^{\alpha n} e^{-3\alpha'' n\left(\frac{2 \tilde C_d \epsilon}{N^d} \log{\frac{2 \tilde C_d \epsilon}{\delta N^d}}+(1-\frac {2 \tilde C_d \epsilon }{N^d} )\log{\frac{1- 2 \tilde C_d \epsilon/N^d}{1-\delta}}\right)} + \mathcal P (n) \nonumber \\
& \quad \,\leq \,  \hat A e^{- \hat B n/2} + C e^{-C' \alpha'' n}  \nonumber\\
& \quad \quad + 4^{\alpha n} e^{-3\alpha'' n\left(\frac{2 \tilde C_d \epsilon}{N^d} \log{\frac{2 \tilde C_d \epsilon}{\delta N^d}}+(1-\frac {2 \tilde C_d \epsilon }{N^d} )\log{\frac{1- 2 \tilde C_d \epsilon/N^d}{1-\delta}}\right)} + 2^{-\frac{n}{N} +1}\,. 
\end{align}
Combining Equations \eqref{eqnewbis1}, \eqref{eqnewbis2} and \eqref{eqnewbis4}, we conclude that for every fixed $\epsilon >0$ and every fixed $p>p_c(d)$ (thus $p_0, q_1$ and $N$ are fixed), there exists $\delta (\epsilon, p) \in (0,q_1 - p]$ and $p_3 (\epsilon, p) >0$ such that for every $ q>p  $ satisfying $q-p<\delta$, for every $x\in\mathbb{S}^1$, for every $n$ large enough, we have
$$\pr  \left[ b_p (\tilde{0}^{\mathcal{C}_p},\widetilde{nx}^{\mathcal{C}_p}) > b_q (\tilde{0}^{\mathcal{C}_q},\widetilde{nx}^{\mathcal{C}_q}) + 12 \alpha'' \tilde{C}_d \rho_d \epsilon n \right] \,\leq \, 1-p_3  \,, $$
thus for every $\epsilon>0$ and for every $p >p_c(d)$, there exists $\delta (\epsilon, p) >0$ such that for every $ q>p $ satisfying $q-p<\delta$, for every $x\in\mathbb{S}^1$,
\[\beta_{p}(x)<\beta_{q}(x)+12 \alpha'' \tilde{C}_d \rho_d \epsilon.\] 
Similarly, for every fixed $\epsilon >0$ and every fixed $q>p_c(d)$ (thus $p_0$ and $N$ are fixed), there exists $\delta' (\epsilon, q) \in (0,q - p_0]$ and $p_4 (\epsilon, q) >0$ such that for every $p<q$ satisfying $q-p<\delta'$, for every $x\in\mathbb{S}^1$, for every $n$ large enough, we have
$$\pr  \left[ b_p (\tilde{0}^{\mathcal{C}_p},\widetilde{nx}^{\mathcal{C}_p}) > b_q (\tilde{0}^{\mathcal{C}_q},\widetilde{nx}^{\mathcal{C}_q}) + 12 \alpha'' \tilde{C}_d \rho_d \epsilon n \right] \,\leq \, 1-p_4  \,, $$
thus for every $\epsilon>0$ and for every $q >p_c(d)$, there exists $\delta' (\epsilon, q) >0$ such that for every $p< q$ satisfying $q-p<\delta$, for every $x\in\mathbb{S}^1$,
\[\beta_{p}(x)<\beta_{q}(x)+12 \alpha'' \tilde{C}_d \rho_d \epsilon.\] 
This ends the proof of Lemma \ref{lem:normcont}.
\end{proof}

\begin{proof}[Proof of Theorem \ref{thm:isocnt}] $\;$\\
\textit{(i) Proof of the continuity of the Cheeger constant}  $\lim_{n\rightarrow\infty}n\varphi_n (p)$. Let $p>p_c(2)$. For any rectifiable Jordan curve $\lambda$, with Leb$(int(\lambda))=1$,
\[
\text{len}_{\beta_p}(\lambda)=\sup_{N\ge 1} \quad \sup_{0\le t_0<\ldots<t_N\le 1}\sum_{i=1}^N\beta_p\left(\frac{\lambda(t_i)-\lambda(t_{i-1})}{\|\lambda(t_i)-\lambda(t_{i-1})\|_2}\right)\|\lambda(t_i)-\lambda(t_{i-1})\|_2.
\]
By Lemma \ref{lem:normcont} for every $\epsilon>0$ there exists a $\delta>0$ such that for every $q>p_c(2)$ satisfying $|p-q|<\delta$ we have $\sup_{x\in\mathbb{S}^1}|\beta_{q}(x)-\beta_p(x)|<\epsilon$, thus
\begin{equation}
\label{eqajout6}
|\text{len}_{\beta_p}(\lambda)-\text{len}_{\beta_{q}}(\lambda)| \leq \epsilon \text{len}_{\|\cdot\|_2}(\lambda)
.
\end{equation}
The infimum in Theorem \ref{thm:cheegervariational} is achieved (by compactess of the set of Lipschitz curves), so let us denote by $\lambda_p$ (resp. $\lambda_q$) a Jordan curve such that Leb$(int(\lambda_p))=1$ and $\text{len}_{\beta_p}(\lambda_p) = \sqrt{2} \theta_p \lim_{n \rightarrow \infty} n \varphi_n(p)$ (resp. Leb$(int(\lambda_q))=1$ and $\text{len}_{\beta_q}(\lambda_q) = \sqrt{2} \theta_q \lim_{n \rightarrow \infty} n \varphi_n(q)$). All norms in $\mathbb{R}^2$ are equivalent thus we know that $\text{len}_{\|\cdot\|_2}(\lambda_p)<\infty$ and $\text{len}_{\|\cdot\|_2}(\lambda_q)<\infty$. From \eqref{eqajout6} we deduce that for every $\epsilon>0$ there exists $\delta>0$ such that if $|p-q|<\delta$ then
\begin{align}
\sqrt{2} \theta_p \lim_{n \rightarrow \infty} n \varphi_n(p) & \,= \, \text{len}_{\beta_p}(\lambda_p) 
  \,\geq \, \text{len}_{\beta_q}(\lambda_p) - \epsilon \text{len}_{\|\cdot\|_2}(\lambda_p) \nonumber\\
& \,\geq \,  \sqrt{2} \theta_q \lim_{n \rightarrow \infty} n \varphi_n(q) - \epsilon \text{len}_{\|\cdot\|_2}(\lambda_p) \label{eqajout7} \\
\text{and } \quad \quad \sqrt{2} \theta_p \lim_{n \rightarrow \infty} n \varphi_n(p) &  \,\leq \, \text{len}_{\beta_p}(\lambda_q)
  \,\leq \, \text{len}_{\beta_q}(\lambda_q) + \epsilon \text{len}_{\|\cdot\|_2}(\lambda_q) \nonumber\\
& \,\leq \,  \sqrt{2} \theta_q \lim_{n \rightarrow \infty} n \varphi_n(q) + \epsilon \text{len}_{\|\cdot\|_2}(\lambda_q) \,.\label{eqajout8}
\end{align}
Let $\beta_q^{\min} = \inf_{x\in \SS^{1}} \beta_q (x)$, for all $q$. By Lemma \ref{lem:normcont} again we know that for every $q$ satisfying $|p-q|<\delta$ we have $\beta_q^{\min} \geq\beta_p^{\min} - \epsilon  $, which is positive for $\epsilon $ small enough ($\beta_p^{\min} $ is not zero since $\beta_p$ is a norm), thus
$$\text{len}_{\|\cdot\|_2}(\lambda_q) \,\leq \, \frac{ \text{len}_{\beta_q}(\lambda_q) }{\beta_q^{\min}}  \,\leq \, \frac{ \text{len}_{\beta_q}(\lambda_q) }{\beta_p^{\min}-\epsilon} \,.$$
Thanks to Equation (\ref{eqajout8}) we obtain
\begin{equation}
\label{eqajout9}
\sqrt{2} \theta_p \lim_{n \rightarrow \infty} n \varphi_n(p) \,\leq \,\sqrt{2} \theta_q  \lim_{n \rightarrow \infty} n \varphi_n(q) \left( 1 + \frac{\epsilon }{\beta_p^{\min}-\epsilon} \right) \,.
\end{equation}
Combining (\ref{eqajout7}) and (\ref{eqajout9}) we obtain that 
$$\lim_{q\rightarrow p}\sqrt{2} \theta_q \lim_{n \rightarrow \infty} n \varphi_n(q) \,=\, \sqrt{2} \theta_p\lim_{n \rightarrow \infty} n \varphi_n(p) \,.$$
Since $p\mapsto \theta_p$ is continuous on $(p_c(2), 1]$, this conludes the first part of the proof.

\medskip
\noindent
\textit{(ii) Proof of the continuity of the Wulff shape.} 
Next we prove that $p\mapsto \widehat{W}_p$ is continuous for the Hausdorff distance. Fix $\eta>0$ and $p>p_c(2)$ and let $\epsilon = \epsilon (\eta, p) >0$ be small enough such that 
\begin{equation}
\label{eqajout11}
\epsilon\,\leq  \,\frac{\beta_{p}^{\min}}{2} \min \left(\eta , 1\right)\,.
\end{equation}
As previously let $\delta >0$ satisfy $\sup_{x\in\mathbb{S}^1}|\beta_{q}(x)-\beta_p(x)|<\epsilon$ for all $q>p_c(2)$ such that $|p-q|< \delta$. For every $x\in W_q$ we have by definition of $W_q$ that for every $\hat{n}\in\mathbb{S}^1$, $\hat{n}\cdot x\le\beta_q(\hat{n})$. Thus for all $q>p_c(2)$ such that $|p-q|< \delta $,
$$\hat{n}\cdot x\le \beta_q(\hat{n})\le\beta_{p}(\hat{n})+\epsilon\le (1+\eta)\beta_{p}(\hat{n}),$$
where the last inequality comes from (\ref{eqajout11}), thus $x\in (1+\eta)W_{p}$. We obtain that for all $p>p_c(2)$, for all $\eta >0$, there exists $\delta >0$ such that for every $q>p_c(2)$ satisfying $|p-q|< \delta $,
\begin{equation}
\label{eqajout10}
W_q\subset (1+\eta)W_{p}.
\end{equation}
For every $q>p_c(2)$ satisfying $|p-q|< \delta $, we also have $\beta_q^{\min} \geq \beta_p^{\min} - \epsilon \geq \beta_p^{\min} /2 \geq  \epsilon / \eta$ by (\ref{eqajout11}), thus by the same method we obtain that for every $x\in W_p$, for every $\hat{n} \in \SS^1$,
$$\hat{n}\cdot x\le \beta_p(\hat{n})\le\beta_{q}(\hat{n})+\epsilon\le (1+\eta)\beta_{q}(\hat{n}),$$
thus 
\begin{equation}
\label{eqajout12}
W_p\subset (1+\eta)W_{q}.
\end{equation}
For every $x \in W_p$, $\|x\|_2 = x \cdot x/\|x\|_2 \leq \beta_p(x) \leq \beta_p^{\max} $, where $\beta_p^{\max} = \sup_{x\in \SS^1} \beta_p(x) <\infty$, thus $\|(1+\eta)x-x\|_2 \leq\eta\beta_p^{\max} $. Similarly, for all $q>p_c(2)$ satisfying $|p-q|< \delta $, $\|x\|_2 \leq \beta_q^{\max} \leq 2\beta_p^{\max} $ and $\|(1+\eta)x-x\|_2 \leq 2 \eta \beta_p^{\max} $. With (\ref{eqajout10}) and (\ref{eqajout12}), we conclude that for every $p>p_c(2)$, for every $\eta >0$, there exists $\delta >0$ such that for every $q>p_c(2)$ satisfying $|p-q|< \delta $, 
$$ d_H (W_p,W_{q} ) \,\leq 2 \eta \beta_p^{\max}  \,,$$
thus $\lim_{q\rightarrow p} d_H (W_p,W_{q} ) = 0$. This implies that $\lim_{q\rightarrow p}\text{Leb}(W_q) = \text{Leb}(W_p)$, and since $\widehat{W}_p=\frac{W_p}{\sqrt{\text{Leb}(W_p)}}$ we deduce from (\ref{eqajout10}) and (\ref{eqajout12}) by a similar argument that $\lim_{q\rightarrow p} d_H (\widehat{W}_p,\widehat{W}_{q} )=0 $. This concludes the proof of Theorem \ref{thm:isocnt}.
\end{proof}

\begin{rem}
To deduce the continuity of the Wulff crystal from Lemma \ref{lem:normcont}, we can also consider a more general setting. Consider $\beta_p^*$ the dual norm of $\beta_p$, defined by
$$ \forall x \in \RR^d \,, \quad \beta_p^* (x) \,=\, \sup \{ x\cdot y \,:\, \beta_p (y) \leq 1 \} \,.$$
Then $\beta_p^*$ is a norm, and what we did is equivalent to deduce from Lemma \ref{lem:normcont} the same result concerning $\beta_p^*$ :
\begin{equation}
\label{eqremarque}
\lim_{q\rightarrow p} \sup_{x\in \SS^1} |\beta_q^*(x) - \beta_p^*(x)| \,=\, 0 \,.  
\end{equation}
Notice that $W_p$, the Wulff crystal associated to $\beta_b$, is in fact the unit ball associated to $\beta_b^*$, then (\ref{eqremarque}) implies the continuity of $p\mapsto W_p$ according to the Hausdorff distance.
\end{rem}


\bigskip
\noindent

\def\refname{References}
\bibliographystyle{plain}
\bibliography{supercontinuite3}

\end{document}

%% file: bypass1.pdf_t
\begin{picture}(0,0)%
\includegraphics{bypass1.pdf}%
\end{picture}%
\setlength{\unitlength}{1973sp}%
\begingroup\makeatletter\ifx\SetFigFont\undefined%
\gdef\SetFigFont#1#2#3#4#5{%
  \reset@font\fontsize{#1}{#2pt}%
  \fontfamily{#3}\fontseries{#4}\fontshape{#5}%
  \selectfont}%
\fi\endgroup%
\begin{picture}(12705,7608)(736,-8557)
\put(1651,-7261){\makebox(0,0)[lb]{\smash{{\SetFigFont{8}{9.6}{\rmdefault}{\mddefault}{\updefault}{\color[rgb]{0,0,0}: the boxes $(\phi_1 (j))_{1 \leq j \leq r_1}$}%
}}}}
\put(1651,-8461){\makebox(0,0)[lb]{\smash{{\SetFigFont{8}{9.6}{\rmdefault}{\mddefault}{\updefault}{\color[rgb]{0,0,0}: the sets of good boxes in $(S_{\varphi_2 (j)})_{1\leq j \leq r_2}$ that do not belong to $(S_{\varphi_4(j)})_{1\leq j \leq r_4}$}%
}}}}
\put(751,-5311){\makebox(0,0)[rb]{\smash{{\SetFigFont{8}{9.6}{\rmdefault}{\mddefault}{\updefault}{\color[rgb]{0,0,0}$y$}%
}}}}
\put(13426,-3061){\makebox(0,0)[lb]{\smash{{\SetFigFont{8}{9.6}{\rmdefault}{\mddefault}{\updefault}{\color[rgb]{0,0,0}$z$}%
}}}}
\put(8626,-2461){\makebox(0,0)[lb]{\smash{{\SetFigFont{8}{9.6}{\rmdefault}{\mddefault}{\updefault}{\color[rgb]{0,0,0}$\gamma$}%
}}}}
\put(1651,-7861){\makebox(0,0)[lb]{\smash{{\SetFigFont{8}{9.6}{\rmdefault}{\mddefault}{\updefault}{\color[rgb]{0,0,0}: the sets of good boxes $(S_{\varphi_4 (j)})_{1\leq j \leq r_4}$}%
}}}}
\end{picture}%

%% file: bypass2.pdf_t
\begin{picture}(0,0)%
\includegraphics{bypass2.pdf}%
\end{picture}%
\setlength{\unitlength}{1973sp}%
\begingroup\makeatletter\ifx\SetFigFont\undefined%
\gdef\SetFigFont#1#2#3#4#5{%
  \reset@font\fontsize{#1}{#2pt}%
  \fontfamily{#3}\fontseries{#4}\fontshape{#5}%
  \selectfont}%
\fi\endgroup%
\begin{picture}(8426,4245)(608,-6073)
\put(6676,-2236){\makebox(0,0)[lb]{\smash{{\SetFigFont{8}{9.6}{\rmdefault}{\mddefault}{\updefault}{\color[rgb]{0,0,0}$\gamma'_{j,link}$}%
}}}}
\put(1951,-2236){\makebox(0,0)[rb]{\smash{{\SetFigFont{8}{9.6}{\rmdefault}{\mddefault}{\updefault}{\color[rgb]{0,0,0}$\gamma'_{j,in}$}%
}}}}
\put(7876,-4036){\makebox(0,0)[lb]{\smash{{\SetFigFont{8}{9.6}{\rmdefault}{\mddefault}{\updefault}{\color[rgb]{0,0,0}$\gamma'_{j,out}$}%
}}}}
\put(8251,-4411){\makebox(0,0)[lb]{\smash{{\SetFigFont{8}{9.6}{\rmdefault}{\mddefault}{\updefault}{\color[rgb]{0,0,0}$\gamma_j$}%
}}}}
\put(826,-2911){\makebox(0,0)[rb]{\smash{{\SetFigFont{8}{9.6}{\rmdefault}{\mddefault}{\updefault}{\color[rgb]{0,0,0}$\gamma_{j+1}$}%
}}}}
\put(5251,-3136){\makebox(0,0)[lb]{\smash{{\SetFigFont{8}{9.6}{\rmdefault}{\mddefault}{\updefault}{\color[rgb]{0,0,0}$\gamma$}%
}}}}
\put(1801,-3436){\makebox(0,0)[rb]{\smash{{\SetFigFont{8}{9.6}{\rmdefault}{\mddefault}{\updefault}{\color[rgb]{0,0,0}$y_j$}%
}}}}
\put(1801,-3886){\makebox(0,0)[rb]{\smash{{\SetFigFont{8}{9.6}{\rmdefault}{\mddefault}{\updefault}{\color[rgb]{0,0,0}$\gamma_{\psi_{in}(j)}$}%
}}}}
\put(3151,-2161){\makebox(0,0)[b]{\smash{{\SetFigFont{8}{9.6}{\rmdefault}{\mddefault}{\updefault}{\color[rgb]{0,0,0}$y'_j$}%
}}}}
\put(7876,-3511){\makebox(0,0)[lb]{\smash{{\SetFigFont{8}{9.6}{\rmdefault}{\mddefault}{\updefault}{\color[rgb]{0,0,0}$x'_j$}%
}}}}
\put(7876,-5761){\makebox(0,0)[lb]{\smash{{\SetFigFont{8}{9.6}{\rmdefault}{\mddefault}{\updefault}{\color[rgb]{0,0,0}$\gamma_{\psi_{out}(j)}$}%
}}}}
\put(7876,-5311){\makebox(0,0)[lb]{\smash{{\SetFigFont{8}{9.6}{\rmdefault}{\mddefault}{\updefault}{\color[rgb]{0,0,0}$x_j$}%
}}}}
\put(3301,-5686){\makebox(0,0)[rb]{\smash{{\SetFigFont{8}{9.6}{\rmdefault}{\mddefault}{\updefault}{\color[rgb]{0,0,0}$S_{\varphi_4(j)}$}%
}}}}
\end{picture}%